\documentclass[a4paper,12pt]{article}

\usepackage{lineno}

\usepackage[T1]{fontenc}
\usepackage{amsmath,amsthm,mathrsfs,amssymb,amsfonts,nicefrac,paralist,subfigure}
\usepackage[english]{babel}
\usepackage[active]{srcltx}
\usepackage[top=3cm, bottom=3cm, left=2.8cm, right=2.8cm]{geometry}
\usepackage{tikz}
\usepackage{dsfont}\usepackage{graphicx}
\usetikzlibrary{patterns}
\usepackage{color}
\usepackage[colorlinks=true]{hyperref}
\newtheorem{e-proposition}[theorem]{Proposition}

\makeatletter

\@addtoreset{equation}{section}
\@addtoreset{chapter}{part}

\newcommand{\RR}{\mathbb{R}}
\newcommand{\N}{\mathbb{N}}

\def\DD{\mathcal{D}}
\def\VV{\mathcal{V}}
\def\UU{\mathcal{U}}
\def\TT{\mathcal{T}}
\def\II{\mathcal{I}}
\def\SST{\mathcal{S}}

\renewcommand{\leq}{\leqslant}
\renewcommand{\geq}{\geqslant}

\renewcommand{\tilde}{\widetilde}

\renewcommand{\t}{\tilde}

\newcommand{\di}{\displaystyle}

\DeclareMathOperator\supp{supp}

\def\1{\mathbbm{1}}
\def\Sph{\mathbb{S}^{N-1}}

\numberwithin{equation}{section}

\newtheorem{Lemma}{Lemma}[section]

\newtheorem{Th}[Lemma]{Theorem}
\newtheorem{Prop}[Lemma]{Proposition}

\newcommand{\be}{\begin{equation}}
\newcommand{\ee}{\end{equation}}
\newcommand{\baa}{\begin{array}}
\newcommand{\eaa}{\end{array}}
\newcommand{\ba}{\begin{eqnarray}}
\newcommand{\ea}{\end{eqnarray}}

\renewcommand\phi{\varphi}

\def\epsilon{\varepsilon}

\def\mc{\mathcal}

\def\trait (#1) (#2) (#3){\vrule width #1pt height #2pt depth #3pt}
\def\fin{\hfill\trait (0.1) (5) (0) \trait (5) (0.1) (0) \kern-5pt \trait (5) (5) (-4.9) \trait (0.1) (5) (0)}
\def\todo#1 {\marginpar{\textcolor{red}{$\Rightarrow$#1}}}
\def\TI{T}
\newcommand{\bee}{\begin{equation*}}
\newcommand{\eee}{\end{equation*}}
\newcommand{\bc}{\begin{cases}}
\newcommand{\ec}{\end{cases}}

\def\e{\varepsilon}

\newenvironment{formula}[1]{\begin{equation}\label{#1}}
{\end{equation}\noindent}

\def\Fi#1{\begin{formula}{#1}}
\def\Ff{\end{formula}\noindent}

\def\cSIR{c_{\scalebox{0.5}{$SIR$}}}
\def\cSIRT{c_{\scalebox{0.5}{$SIR$}}^{\scalebox{0.5}{$T$}}}
\def\wSIR{w_{\scalebox{0.5}{$SIR$}}}
\def\wSIRT{w_{\scalebox{0.5}{$SIR$}}^{\scalebox{0.5}{$T$}}}
\def\WSIRT{\omega_{\scalebox{0.5}{$SIR$}}^{\scalebox{0.5}{$T$}}}

\begin{document}
\title{\bf Propagation of epidemics along lines with fast diffusion}
\author{Henri {\sc Berestycki}$^{\hbox{a,b}}$,
Jean-Michel {\sc Roquejoffre}$^{\hbox{c}}$, Luca {\sc Rossi}$^{\hbox{a}}$\\
\footnotesize{$^{\hbox{a }}$ Ecole des Hautes Etudes en Sciences Sociales,  CNRS}\\
\footnotesize{Centre d'Analyse et Math\'ematiques Sociales,}\\ 
\footnotesize{54 boulevard Raspail, F-75006 Paris, France}\\
\footnotesize{$^{\hbox{b }}$ Senior Visiting Fellow, HKUST Jockey Club Institute for Advanced Study, }\\
\footnotesize{Hong Kong University of Science and Technology}\\
\footnotesize{$^{\hbox{c }}$ Institut de Math\'ematiques de Toulouse,
Universit\'e Paul Sabatier}\\
\footnotesize{118 route de Narbonne, F-31062 Toulouse Cedex 4, France}
}
\maketitle

\vspace{.2cm}

\begin{abstract} It has long been known that epidemics can travel along communication lines, such as roads. In the  current COVID-19 epidemic, it has been observed that major roads have enhanced its propagation in Italy. We propose a new simple model of propagation of epidemics which exhibits this effect and allows for a quantitative analysis. The model consists of a classical $SIR$ model with diffusion, to which an additional compartment is added, formed by the infected individuals travelling on a line of fast diffusion.
Exchanges between individuals on the line and in the rest of the domain are taken into account. 
A classical transformation allows us to reduce the proposed model to a system analogous to one  we had previously introduced \cite{BRR2} to describe the enhancement of biological invasions by lines of fast diffusion. We establish the existence of a minimal spreading speed and we show that it may be quite large, even  when the basic reproduction number~$R_0$ is close to $1$. More subtle qualitative features of the final state, showing the important influence of the line, are also proved here.
 \end{abstract}

\hypersetup{linkcolor=black}
\hypersetup{linkcolor=red}


\section{Context and motivation of this study}\label{mot}

In the present context of the COVID-19 pandemic, a worldwide scientific effort is currently under way to develop the modelling of its dynamics and propagation. Such an endeavour is of essential value to monitor, and forecast the propagation of the epidemic.

Most of the models that are used rely on various extensions of the classical $SIR$ cornerstone model of epidemiology. That is, they use population compartmental models that contain
various additional compartments to the $SIR$ ones to account for segments of the populations that are exposed, asymptomatic, presymptomatic, treated,~etc. Such models yield the evolution of the infected population at a given level of territorial granularity (whole countries, regions, counties or cities). The spatial interplay aspect, mostly overlooked, is included by involving transfer matrices of populations and infected between various patches each of which being considered as uniform. 

Yet, the propagation of COVID-19 exhibits remarkable spatial structure properties. 
Indeed, the spatial organization and spreading of epidemics in general reveal important features of the transmission process. This is especially true in relation with movements of individuals who carry with them infectious characteristics. It has long been known,
since ancient times, that epidemics travel along lines of communications. In the black death epidemic of the 14th century, contagion advanced along roads connecting trade fair cities,
and from there spread inwards leading to a front like invasion of Western Europe roughly from South to North,  pulled by these roads. Such a mode of propagation was also at work for the propagation of rumours. (See for instance the propagation of the ``big fear'' in France, after the Revolution.
Compare, for example, the presentation given by Siegfried in \cite{Sieg} and the analogies between these two phenomena.)
In studying the early spread of HIV virus in Congo \cite{Congo}, Faria et al.~pointed out that the virus mostly travelled along the main lines of communications (railways and waterways), see Figure~\ref{fig:HIV} below.  
\begin{figure}[h]
	\begin{center}
		\includegraphics[width=.5\textwidth]{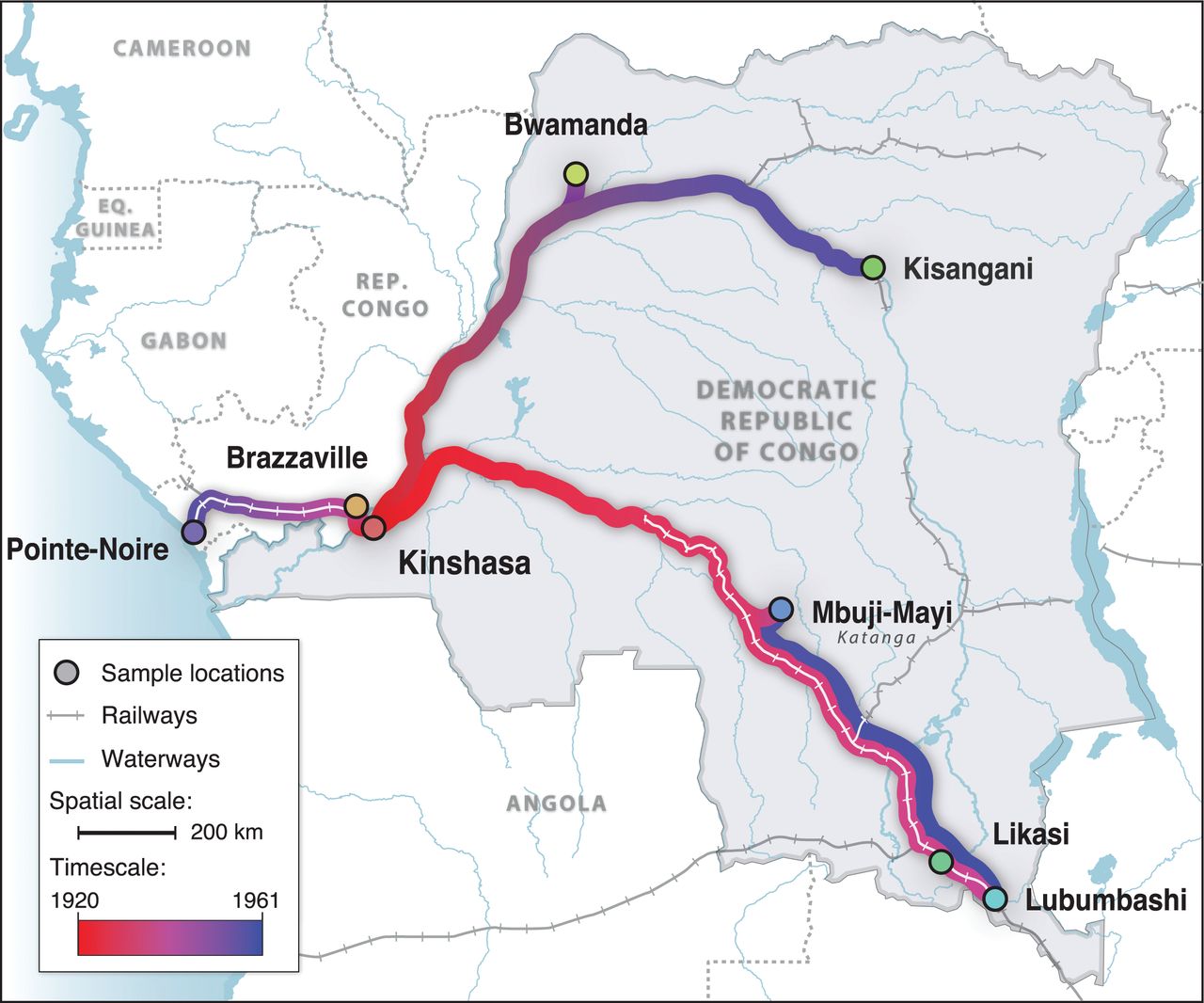}
		\caption{From Faria et al.~\cite{Congo} in {\it Science}.
			Strongly supported rates of virus spatial movement are projected along the transportation network for the Democratic Republic of Congo (railways and waterways). {\it Reprinted with permission of AAAS.}}
		\label{fig:HIV}
	\end{center}
\end{figure}

Some current studies are bringing to light a similar effect in the spread of the COVID-19 virus in Italy.  
Gatto et al.~\cite{COVID-Gatto} and Sebastiani~\cite{COVID-Sebastiani} have established that the coronavirus spread foremost along the main expressways. They argue indeed that cities located along the main North-South and East-West highways in Italy have faced an earlier and stronger contagion than
cities with similar or higher population sizes but not located on these roads. 
The spatial signatures of the spreading of the epidemic is clearly revealed by
Figure~\ref{fig:Gatto_Fig1} taken from the work~\cite{COVID-Gatto} by 
Gatto et al.
\begin{figure}[h]
	\begin{center}
		\includegraphics[width=\textwidth]{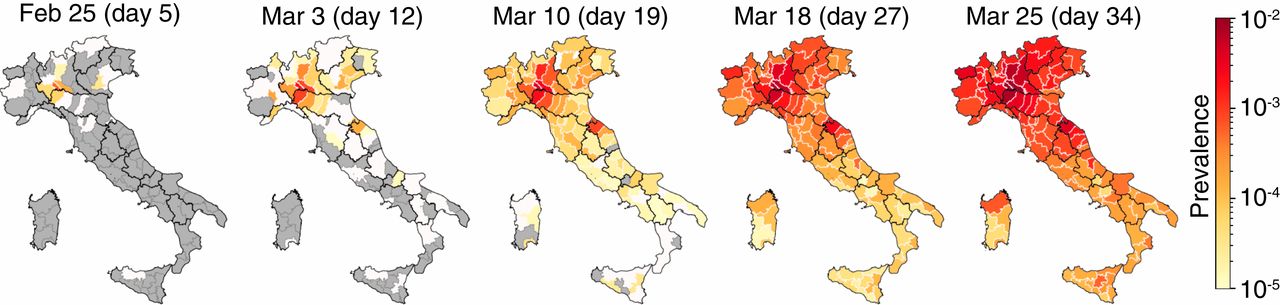}
		\caption{From Gatto et al.~\cite{COVID-Gatto} in {\it PNAS}. Spread of COVID-19 in Italy at the province level, spanning the time period from Feb.~25 to Mar.~25. The quantity represented is the ratio between the total number of confirmed cases 
			and the resident population. {\it Creative Commons Attribution License 4.0 (CC BY).}}
		\label{fig:Gatto_Fig1}
	\end{center}
\end{figure}

Then, Figure~\ref{fig:Gatto_Fig3} is a snapshot from a video in~\cite{COVID-Gatto} which accounts for 
hospitalisation. It highlights the
radiation of the epidemic along highways and transportation infrastructures.
Figure~\ref{fig:Sebastiani} 
below, taken from Sebastiani~\cite{COVID-Sebastiani}, 
depicts the Italian cities with more than 1,000 infected individuals reported as for April 5. 
For two of them, Piacenza and Cremona, which are located respectively 
at the crossroad of these two main highways and 40 Km away from it,
infected people are above 1\% of the population.
We point out that the cities represented in this figure are by far not the most populated ones in Italy.
%
\begin{figure}[h]
	\begin{center}
		\includegraphics[width=\textwidth]{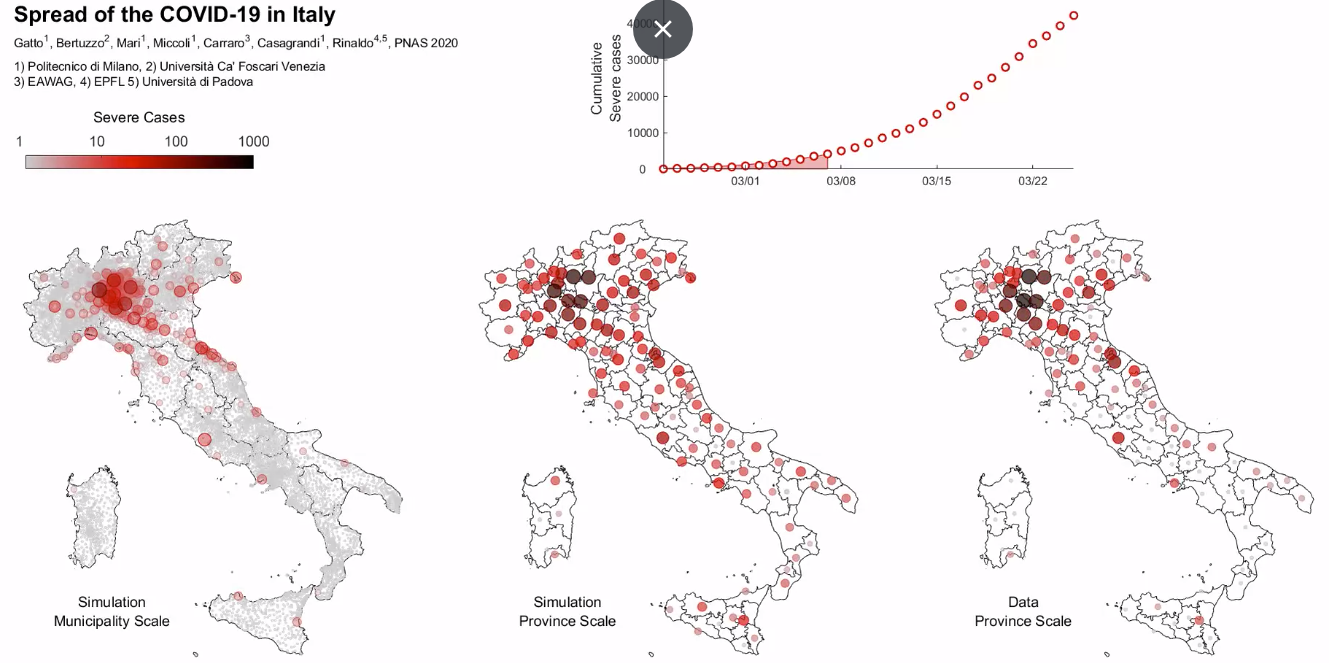}
		\caption{From Gatto et al.~\cite{COVID-Gatto} in {\it PNAS}. 
			Comparison between simulated and recorded
			cumulative number of severe cases that required hospitalization as for March 5: 
			model simulation at the municipality scale (left) and province scale (centre); 
			recorded data at the province scale (right).
			{\it Creative Commons Attribution License 4.0 (CC BY).}}
		\label{fig:Gatto_Fig3}
	\end{center}
\end{figure}

That this effect of roads still matters in such a global and ``modern'' epidemic as COVID-19 shows that spatial diffusion still plays an important role in the spreading of epidemics.
At an early stage, 
long distance ``jumps'' through the air transport network played the key role in the global dissemination of the pandemic. However,
the role of ground transportation became relevant in a second phase. This is even more true after the lockdown in many countries, that  all but  halted all flights.
Even though ground travel was limited, it was never interrupted and, furthermore, ground transportation was essential in transporting all needed supplies.

\begin{figure}[h]
	\begin{center}
		\includegraphics[height=5.5cm, width=6.5cm]{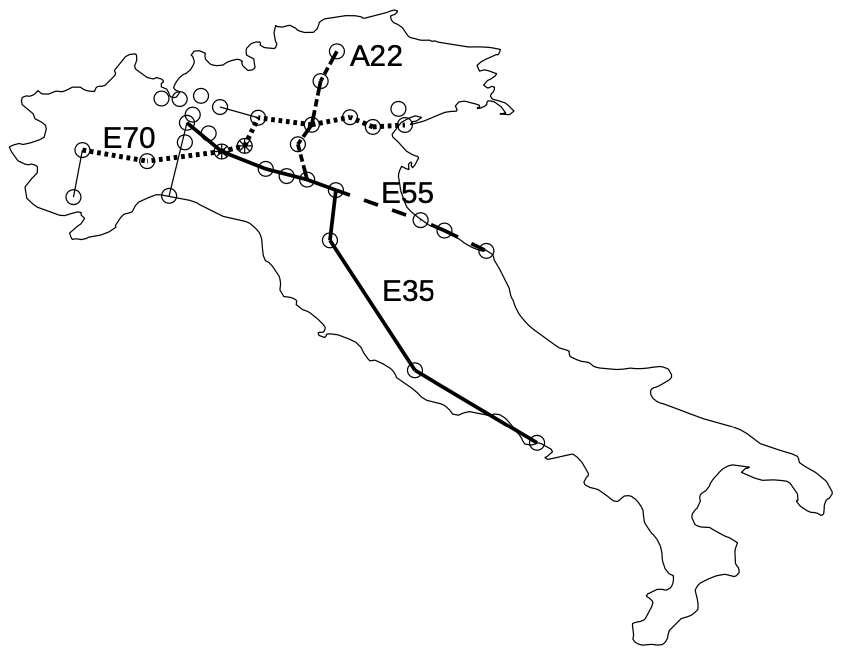}
		\caption{From Sebastiani~\cite{COVID-Sebastiani}. The 33 Italian cities where, as of April 5, the total number of 
			infected people reported is larger than 1,000; the streaked circles
			correspond to Piacenza and Cremona, where infected people are above 1\% of the population.
			{\it Reproduced with permission of G.~Sebastiani.}}
		\label{fig:Sebastiani}
	\end{center}
\end{figure}


With respect to patch models that include transfer matrices, it thus transpires that the effect of lines is of a different nature. And as we have seen, these lines can be roads, railroads or rivers.

The aim of this paper is to propose a new model to account for such effects and then to study quantitatively how a line acts on the overall epidemics propagation. We thus introduce a model that we call the $SIRT$ system, 
standing for {\bf S}usceptibles, 
{\bf I}nfected, {\bf R}ecovered, and {\bf T}ravelling infected. This model takes explicitly into account the existence of a line along which infected individuals can travel with a specific diffusion coefficient. 
Our aim here is to gain insight into this spreading aspect at the fundamental mathematical level of a $SIR$ type model that now incorporates the possibility of infected to travel along a specific line.

The $SIR$ model dates back the fundamental paper of  Kermack-McKendrick~\cite{KMK}, 
Kendall~\cite{Kend2}  introduced the $SIR$ model with spatial interaction in the discussion of a statistical study of measles by Bartlett~\cite{Bart}. 
These models may be rewritten as nonlinear integro-differential equations in time and space variables, for which the study of spreading  dates back to the 1970's with, in particular, the milestone papers of Aronson \cite{A} and Diekmann \cite{Diek}. 
We choose in this paper to restrict our model to local (Brownian) diffusion and local interaction. 
Regarding the existence of travelling waves for the local (homogeneous) diffusion model, see the studies of K\"allen \cite{Ka}, Hosono-Ilyas  \cite{HI2, HI}. For a survey, see Murray \cite{Murray}.  

The model we propose here aims at a fundamental mathematical understanding of this effect.  We look at a stylized situation -- a laboratory case as it were. The population of susceptibles diffuse  
in a homogeneous
open territory, that we take to be a half plane. Then, the infected can travel along the line bounding this half plane. The populations of infected in the open territory and on the lines are 
in constant exchange. This is represented in our model by transmission parameters between these two populations (stable and travelling infected). We aim at understanding the effects of 
such lines on the speed of propagation of the epidemics, and how they affect the balance of the
total number and locations of infected individuals. 

As a benchmark we use the classical $SIR$ model with diffusion, for which  we briefly recall the basic results and 
how to compute the speed of spreading. 
The latter coincides with that of  
Aronson-Weinberger \cite{AW}, we simply rewrite the formulas   in terms of the basic reproduction number~$R_0$.   

Let us mention that the $SIR$ model falls into a more general class of systems in which one equation
--~the one for $I$~here~-- exhibits a
self-reinforcement mechanism when the other unknown function --~$S$~-- 
is above some threshold level. 
Berestycki-Nordmann-Rossi~\cite{BNR-SIR} develop a
mathematical study of this class of systems,
called {\em Activity/Modulator}. 
As described there, $SIR$ models also arise 
in a variety of contexts to model contagion phenomena. 
They are particularly useful in the study of collective behaviours,
such as social unrest. The work by Bonasse-Gahot et al.
\cite{Bonnasse-Gahot2018}, developed such a system to model the spread of riots in France in 2005.


Of course, more realistic epidemiology models will incorporate networks rather than single lines and, likewise,
the remaining propagation does not take place in a homogeneous territory.
From this simplified model, one can nonetheless deduce a more realistic one involving a network of roads and a distribution of cities. Regarding the latter aspect, Bonasse-Gahot et al.~\cite{Bonnasse-Gahot2018} 
use explicitly such a network of cities. 

This family of models lends itself to various extensions. But since we want to derive the mathematical properties of this system we only consider here the stylized model.
We hope that this stylized model can shed some light on how lines influence the global unfolding of an~epidemic.

\subsection*{Acknowledgment}
The authors are thankful to Italo Capuzzo-Dolcetta for bringing to their attention the works of Gatto et al.~\cite{COVID-Gatto} and Sebastiani~\cite{COVID-Sebastiani} on the role of major roads on the propagation of COVID-19 in Italy.


\section{A model for the propagation of epidemics along lines}\label{s1}

\subsection{The $SIR$ model with infected diffusion}
The classical $SIR$ model divides the overall population in three compartments: Susceptibles $S$, Infected (and infectious) $I$ and recovered $R$. Since the total population $N$ is viewed as fixed, one derives the latter in a straightforward manner from the former two: $ R= N- S-I$. Therefore, we do not mention explicitly this function henceforth. When taking into account spatial dependence, we view the unknowns $S$ and $I$ as depending on both time $t$ and space location $X\in\RR^2$.
We consider that the individuals in the susceptible population $S$ do not move around (or rather that movement does no affect its distribution). We can think of $S$ as the ambient population. 
Thus, we assume that only the infected population $I$ is subject to movement. We choose here to represent this movement as a pure local diffusion that can be viewed as a limiting Brownian movement of individuals. In the second part of this work \cite{BRRnonloc}, we consider the case of non-local diffusions. 
We denote by $d$ the diffusion coefficient. We are thus led to the following spatial diffusion $SIR$ model:
\begin{equation}\label{eSIR}
\left\{ \begin{array}{ll} 
	\partial_tI-d\varDelta I=\beta SI -\alpha I & \ (t>0,\ X\in \mathbb {R}^2) \\ 
	\partial_tS=-\beta SI & \ (t>0,\ X\in \mathbb {R}^2).\\ 
\end{array} \right.
\end{equation}
The system is supplemented with initial conditions. We assume that the initial distribution of 
$S$, $S(0, X)\equiv S_0$ is constant and that $I(0,X)= I_0(X)$  is compactly supported.  The cumulative number of infected at location $X$, at time $t$ is given by $\int_0^{t} I(t, X) dt$.
Several authors have considered this model or closely related ones with integral formulations. We mention some of the references in Subsection~\ref{sec:SIR} below. For the sake of completeness we derive here the main results concerning this system \eqref{eSIR}: existence and properties of a final state to which the solutions and the cumulative number of 
infections at each location converge, characterization of epidemic spreading, and asymptotic speed of propagation. These properties involve the classical basic reproduction number 
$R_0:= S_0 \beta / \alpha$. As we will recall, the position of $R_0$ with respect to~$1$ determines here too a threshold for the epidemic to spread. These various properties will serve as benchmarks for our study of the effect of the presence of a road.

\subsection{A ``$SIRT$'' model in the presence of a road}

Let us now introduce this system, which we call a {\em road/field} model. It represents a situation where there is a {\em road} on which infected individuals can travel with a specific diffusion coefficient $D$. We are especially interested in the case of a large $D$. The aim is to understand in what way such a road alters the epidemic spreading, in particular if it enhances its propagation and in what measure exactly. 
The starting point of our analysis 
consists in distinguishing the infected individuals which are moving on the line with fast diffusion 
from the ones present in the rest of the territory, by using two distinct density functions.
This is the same modelling hypothesis we made in our previous paper~\cite{BRR2}
to describe the dynamics of a single population in the presence of a road.

As before we use the densities $S(t,x,y)$ and $I(t,x,y)$ of susceptible and infected individuals at time $t\geq0$ and position $X=(x,y)\in\RR^2$. We still assume that we can ignore the movement of susceptibles.
We introduce a new compartment of the population that we call $\TI(t,x)$ standing for ``travelling individuals''. This is the density of infected individuals on the road, that is the line $\RR$.
It is worth emphasizing that $I(t,x,y)$ and $\TI(t,x)$ are two different compartments, in particular  $\TI(t,x)$ is not the same as $I(t,x, 0)$. 
Both populations are assumed to diffuse, but with different diffusion coefficients: $D$ for $\TI$ on the line and 
$d$ for $I$ in the plane. The two compartments interact by a constant exchange of  individuals: at each time $t>0$ and point $x\in\RR$, 
the line  yields an amount $\mu \TI(t,x)$ of individuals to the domain, 
and receives an amount $\nu I(t,x,0)$ of individuals from the domain. 

To shed light on the effect of such a road it is enough to consider the interplay of a half-plane with the road that bounds it. Indeed, results in this framework can be translated for results in the whole space by straightforward symmetry arguments. But it can also be seen that propagation properties in the plane in general can easily be deduced from the half-plane case. (For such a discussion, we refer to our earlier paper for the KPP equation with a road \cite{BRR2}.)
We prefer to carry our analysis in the half-plane case for the sake of clarity and to simplify notations. 

Thus, we consider the following system for the unknown $S, I, \TI$: 
By symmetry, we can reduce to study the problem in the upper half-space $\RR\times(0,+\infty)$.
We end up with the following system :
\begin{equation}
\label{e1.1}
\left\{ \begin{array}{ll} \partial_tI-d\varDelta I+\alpha I=\beta SI & (t>0,\ x\in \mathbb {R},\ y>0)\\ 
	\partial_tS=-\beta SI & (t>0,\ x\in \mathbb {R},\ y>0)\\ 
	- d\partial_yI=\mu T-\nu I & (t>0,\ x\in \mathbb {R},\ y=0)\\ 
	\partial_t T-D\partial_{xx}T=\nu I(t,x,0)-\mu T \ & (t>0,\ x\in \mathbb {R},\ y=0). \end{array} \right.
\end{equation}
We assume a uniform initial susceptible density: $S(0,x,y)\equiv S_0>0$. 
The initial density of infected individuals is assumed to be zero outside a bounded region: $$(\TI(0,x),I(0,x,y))=(\TI_0(x),I_0(x,y))$$ 
are nonnegative and compactly supported in $\RR$ and $\RR\times[0,+\infty)$ respectively,
with in addition $I_0\not\equiv0$, but possibly $T_0\equiv0$. Actually, to simplify matters, most fo the time we will look at the case $T_0\equiv0$, not that it changes much in the proofs.

\subsection{Organization of the paper}
In the next section, we will first analyse systems \eqref{eSIR} to have sound benchmarks that will be useful to discuss the effects of the road. Then, the remaining of the paper will be devoted to the study of \eqref{e1.1}. We will start by discussing the stationary limiting state and see whether the epidemic spreads or not in Subsection~\ref{sec:SIRT}. There we will also derive the asymptotic speed of propagation for the spatial spread of the epidemic. In Section~\ref{sec:effects}, we discuss how the presence of the road affects the distribution of the total number of infected per location.  Section~\ref{sec:proofs} is devoted to the mathematical proofs of these various results. These
depend on the various parameters and we discuss their influence in Section~\ref{sec:parameters}.


\section{Analysis of the models}\label{sec:analysis}

We are going to compare the behaviour of \eqref{e1.1} to that of the classical $SIR$ model in the whole plane, with no diffusion of the susceptible population, that is~\eqref{eSIR}. 
The proofs of the results (in a slightly broader framework)
are presented in Section~\ref{sec:proofSIR}.


\subsection{Benchmark: the $SIR$ model}\label{sec:SIR}

System \eqref{eSIR} can be reduced to the classical Fisher-KPP equation by noticing that $S(t,X)$ is easily computed from $I(t,X)$:
\begin{equation}
\label{e2.2000}
\mathrm{ln}\biggl(\frac{S(t,X)}{S_0}\biggl)=-\beta v(t,x),\ \ v(t,x)=\int_0^tI(s,X)ds.
\end{equation}
Thus, the function 
$v(t,X)$ satisfies the equation
\begin{equation}
\label{e2.2}
v_t-d\Delta v=f(v)+I_0(X)\qquad(t>0,\ X\in\RR^2)\\
\end{equation}
with
\begin{equation*}
\label{f(u)}
f(v):=S_0(1-e^{-\beta v})-\alpha v,
\end{equation*}
together with the initial condition
\begin{equation}
\label{v0}
v(0,X)\equiv0 \qquad(X\in\RR^2).
\end{equation}
Transform \eqref{e2.2000} is a well-known particular case of a broader class (see for instance Aronson~\cite{A}, 
Diekmann~\cite{Diekmann-78}) that reduces the $SIR$ model with nonlocal interactions, and with no diffusion on the susceptible individuals, to 
nonlinear integral equations. In this particular case, we retrieve a parabolic equation. Pulsating waves and spreading speeds for periodic $S_0$ are studied for \eqref{eSIR} by Ducrot-Giletti~\cite{DG-SIR}. The integral equation resulting from the model with nonlocal interactions is studied by Ducasse~\cite{Duc}.

The nonlinearity $f$ is concave and vanishes at $0$, it is then of the 
KPP type. Only the presence of the ``source'' term $I_0$ differs 
from the standard Fisher-KPP equation. 
However, being $I_0$ compactly supported, the dynamics of the equation is still governed 
by the sign of $f'(0)$. Using a notation commonly employed in the literature, 
we write
$$f'(0)=\alpha(R_0-1),\quad\text{where }\;R_0:=\frac{S_0\beta}\alpha.$$
The quantity $R_0$ can be viewed as the classical {\em basic reproduction number},
see for instance~\cite{Ka} and the discussion on such number and its
interpretation in~\cite{Diekmann-R0}.
This parameter plays a crucial role and is widely mentioned 
by experts, decision makers and media for the analysis of the present COVID-19 pandemic.   
The sign of $f'(0)$ is determined by the position of $R_0$ with respect to
the threshold value $R_0=1$.
The presence of $I_0$ in~\eqref{e2.2} prevents the existence of constant steady states,
making the dynamics of the equation more complex. Nevertheless, the following  
{\em Liouville-type} result holds.
\begin{Th}\label{thm:Liouville}
	The equation in~\eqref{e2.2} admits a unique positive, bounded, stationary solution~$v_\infty(X)$.
	Moreover, $v_\infty$ satisfies
	\[\lim_{|X|\to\infty}v_\infty(X)=
	\begin{cases} 
	0 & \text{if }\,R_0\leq1\\
	v_* &\text{if }\,R_0>1,
	\end{cases}
	\]
	where $v_*$ is the unique positive zero of $f$.
\end{Th}

The proof of this result, as well as the others in this section, are presented in 
Section~\ref{sec:proofSIR}. They can also be found in Ducrot-Giletti~\cite{DG-SIR},
at least for the case $R_0>1$,
in the more general framework of periodic coefficients $\alpha(X),\beta(X)$ and distribution $S_0(X)$.
Some additional qualitative properties of~$v_\infty$ are contained in Theorem~\ref{thm:further} below.
Namely, $v_\infty$ is radially decreasing outside the support of $I_0$ and it has the shape of a {\em bump} above 
the value $0$ or $v_*$.

It turns out that the steady state $v_\infty$ is a global attractor for the dynamics 
of~\eqref{e2.2}.

\begin{Th}\label{thm:ltb}
	The solution $v(t,X)$ to~\eqref{e2.2}-\eqref{v0} converges locally uniformly to 
	$v_\infty(X)$ as $t\to+\infty$.
\end{Th}
Since the above convergence holds true for the time derivatives (see the proof in Section~\ref{sec:proofSIR}),
a first consequence we derive is: 
$$I(t,X)=\partial_t v(t,X)\to0\quad \text{as }t\to+\infty,$$
that is, the number of infected individuals drops to $0$  asymptotically in time. 

One can also read Theorem~\ref{thm:ltb} in terms of number of
remaining susceptibles at time $t$ and position $X$, which is 
$S(t,X)=S_0e^{-\beta v(t,X)}$. 
Hence, the amount of people that will be infected by
the virus at a given place $X$, throughout the whole course of the epidemic, is given by
$$I_{tot}(X)=S_0\big(1-e^{-\beta v_\infty(X)}\big).$$ 
We point ou that the fact that this function is not constant is a consequence 
of the hypothesis that susceptibles do not diffuse.
Otherwise the diffusion tends to ``flatten'' the density $S$, and this mechanism occurs  
for rather general systems, see e.g.~\cite[Theorem~4.3]{BNR-SIR}.
For this reason, the description of the function $v_\infty$ is extremely important in 
the study of the epidemic.
With this regard, Theorem~\ref{thm:Liouville} leads to the following crucial 
dichotomy for the total amount of infected people far from the epicentre:
$$
\lim_{|X|\to\infty}I_{tot}(X)=\begin{cases} 
0 & \text{if }R_0\leq1\\
S_0\big(1-e^{-\beta v_*}\big) & \text{if }R_0>1.
\end{cases}
$$
%
%
This corresponds to two opposite scenarios.
If $R_0\leq1$, the epidemic does not propagate and areas very far from 
the initial outburst (the support of $I_0$) will be essentially not infected.
Conversely, if $R_0>1$, the epidemic propagates across the territory, and, even though
its impact will be weaker on places far from the epicentre,
the total infected people there will be a portion $1-e^{-\beta v_*}$
of the overall population.

What is important to determine in the case $R_0>1$ is the speed at which the
epidemic spreads. This is provided by the following.

\begin{Th}\label{thm:speed}
	
	
	Assume that $R_0>1$. Call
	\begin{equation}
	\label{e2.3}
	\cSIR:=2\sqrt{d\alpha(R_0-1)}.
	\end{equation}
	Then, for all $\e\in(0,\cSIR)$, 
	the solution $v(t,X)$ to~\eqref{e2.2}-\eqref{v0} satisfies
	$$\lim_{t\to+\infty}\Big(\max_{\vert X\vert\leq(\cSIR-\e)t}|v(t,X)-v_\infty(X)|\Big)=0,
	$$
	$$\lim_{t\to+\infty}\Big(\max_{\vert X\vert\geq(\cSIR+\e)t}v(t,X)\Big)=0.
	$$
	
\end{Th}

The quantity~$\cSIR$ is the {\em asymptotic speed of spreading}
of the epidemic wave. 
It coincides from one hand with 
the speed $2\sqrt{df'(0)}$ of the standard Fisher-KPP equation, and from the other
with the minimal speed of the waves for the $SIR$ model~\eqref{eSIR}
obtained by K\"{a}ll\'{e}n~\cite{Ka}. 

A word must be said about the proof, presented in Section~\ref{sec:proofSIR}.
It is essentially a direct consequence of Theorem~\ref{thm:ltb} above and
Aronson-Weinberger \cite{AW}. Indeed, the former describes the behaviour of the solution 
in compact regions, while the latter provides the estimates far from the origin,
where the influence of $I_0$ is negligible.

\subsection{Dynamics of the model in the presence of the line}\label{sec:SIRT}

\noindent It occurs that the very same transform works for the model with the line~\eqref{e1.1}:
$$
u(t,x):=\int_0^t\TI(s,x)ds,\qquad v(t,x,y):=\int_0^tI(t,x,y)ds.
$$
The system for $(u,v)$ is
\begin{equation}
\label{e2.4}
\left\{ \begin{array}{ll} \partial_t u-D \partial_{xx} u=  \nu v(t,x,0)-\mu
	u+T_0(x) & (t>0,\ x\in \mathbb {R})\\ \partial_t v-d\varDelta v=f(v)+I_0(x,y) & (t>0,\ x\in
	\mathbb {R},\ y>0)\\ -d\partial_y v(t,x,0)=\mu u(t,x)- \nu v(t,x,0) & (t>0,\ x\in \mathbb {R}).
\end{array} \right.
\end{equation}
with, as before, 
$f(v):=S_0(1-e^{-\beta v})-\alpha v$.
The initial datum is 
\begin{equation}\label{u0v0}
(u(0,\cdot),v(0,\cdot))\equiv(0,0).
\end{equation}

This is the same system studied in~\cite{BRR2}, except for the ``source'' terms $I_0,T_0$.
We recall that these are assumed here to be nonnegative and compactly supported, with 
$I_0\not\equiv0$ (and typically $T_0\equiv0$).

We start with the Liouville-type result, analogous to Theorem~\ref{thm:Liouville}.

\begin{Th}\label{thm:Liouville-SIRT}
	System~\eqref{e2.4} admits a unique positive, bounded, stationary 
	solution $(u_\infty^r,v_\infty^r)$.
	Such solution satisfies
	\[
	\lim_{|x|\to\infty}\big(u_\infty^r(x),v_\infty^r(x,y)\big)=
	\begin{cases} 
	(0,0) & \text{if }R_0\leq1\\
	\big(\frac\nu\mu,1\big)v_* & \text{if }R_0>1,
	\end{cases}
	\quad\text{uniformly in }y\geq0,
	\]	
	\[
	\lim_{y\to+\infty}v_\infty^r(x,y)=
	\begin{cases} 
	0 & \text{if }R_0\leq1\\
	v_* & \text{if }R_0>1,
	\end{cases}
	\quad\text{uniformly in }x\in\RR,
	\]
	where $v_*$ is the unique positive zero of $f$.
\end{Th}	
The long-time behaviour for~\eqref{e2.4}
is described by the following result, which is the counterpart of Theorem~\ref{thm:ltb}
about the model without the line.

\begin{Th}\label{thm:ltb-SIRT}
	The solution $(u,v)$ to~\eqref{e2.4}-\eqref{u0v0}
	converges to $(u_\infty^r,v_\infty^r)$
	as $t\to+\infty$, locally uniformly in $x\in\RR$, $y\geq0$.
\end{Th}

At this stage, the picture is identical to the standard $SIR$ model described in the previous section.
The total amount of infected individuals after the passage of the 
epidemic wave,  at a given point $(x,y)$, is 
$$I_{tot}(x,y)=S_0\big(1-e^{-\beta v_\infty^r(x,y)}\big),$$ 
with $v_\infty^r$ exhibiting two qualitatively distinct behaviours depending on whether
$R_0\leq1$ or $R_0>1$. Namely, 
the total number of infected people very far from the epicentre of the epidemic is 
$$
\lim_{|(x,y)|\to\infty}I_{tot}(x,y)=\begin{cases} 
0 & \text{if }R_0\leq1\\
S_0\big(1-e^{-\beta v_*}\big) & \text{if }R_0>1,
\end{cases}
$$
which is the same as in the case without the line.

Next, we investigate the speed at which the epidemic spreads across the territory, along the line.
\begin{Th}\label{thm:speed-SIRT}
	Assume that $R_0>1$. Let $(u,v)$ be the solution to~\eqref{e2.4}-\eqref{u0v0}. 
	Then, there exists $\cSIRT>0$ such that, for all $\e\in(0,\cSIRT)$,
	$$\lim_{t\to+\infty}\Big(\max_{\vert x\vert\leq(\cSIRT-\e)t}\big|(u(t,x),v(t,x,y))-
	(u_\infty^r(x),v_\infty^r(x,y))\big|\Big)=0,
	$$
	$$\lim_{t\to+\infty}\Big(\max_{\vert x\vert\geq(\cSIRT+\e)t}\big|(u(t,x),v(t,x,y))\big|\Big)=0,
	$$
	locally uniformly with respect to $y\geq0$.
	
	In addition, the spreading speed $\cSIRT$ satisfies
	$$\cSIRT\begin{cases} = \cSIR & \text{if }D\leq2d\\
	> \cSIR & \text{if }D>2d.
	\end{cases}$$		
\end{Th}

The {\em asymptotic speed of spreading} $\cSIRT$ coincides with the one for the 
homogeneous model, i.e.~with $I_0,T_0\equiv0$, which is provided by 
\cite[Theorem~1.1]{BRR2}.
Form a mathematical point of view, this means that the presence of the compact 
perturbation does not affect the speed of propagation, as in the case of the $SIR$ model~\eqref{e2.2}.
The fact that the speed cannot decrease if one adds the perturbation is a straightforward consequence 
of the comparison principle. Instead, to derive the opposite inequality, we need to go 
into the proof of \cite[Theorem~1.1]{BRR2} and use the supersolutions constructed there.   
This is done in Section~\ref{sec:proofSIRT} below.


Theorem~\ref{thm:speed-SIRT}
shows that  the presence of the line has a true impact on the speed at which the epidemic 
spreads. Indeed, if the diffusion coefficient on the line $D$ is larger than twice the one in the field $d$,
the asymptotic speed of spreading in the direction of the line
is enhanced, compared with the standard one $\cSIR$. How much 
the spreading is enhanced is discussed in Section~\ref{sec:parameters}.
One can then wonder what is the effect of the line on the other directions. In the case $I_0,T_0\equiv0$,
we have shown in \cite{BRR-directions} that the enhancement of the speed occurs in a cone of directions
around the line, with an associated critical angle.
We suspect the same scenario to hold true for system~\eqref{e2.4}.

We conclude this section by showing 
that our model accounts for a true epidemic wave, in the following sense. At every point of the domain under consideration, the number of infected individuals, that was initially close to 0, raises to a nontrivial level around a certain time, then decays back to 0.
\begin{Prop}\label{p4.2}
	Assume that $R_0>1$. 
	There is a constant $T_*>0$, a function $I_*(y)$, defined for $y\geq0$ and locally bounded 
	from below away from 0, and a function $\tau_*(x)$, defined for $x\in\RR$ and such that 
	\begin{equation}
	\label{tau}
	\lim_{\vert x\vert\to+\infty}\frac{\tau_*(x)}{\vert x\vert}=\frac1{\cSIRT},
	\end{equation}
	for which the following is true.
	\begin{enumerate}
		\item (peak around $\tau_*(x)$). We have 
		\begin{equation}
		\label{e4.100}
		T(\tau_*(x),x)\geq T_*, \ \ \  I_*(\tau_*(x),x,y)\geq I_*(y).
		\end{equation}
		\item (decay far from $\tau_*(x)$). The following limits hold uniformly in 
		$x\in\RR$ and locally uniform in $y\geq0$:
		\begin{equation}
		\label{e4.101}
		\begin{array}{rll}
		\di\lim_{t\to+\infty}\biggl(T(\tau_*(x)+t,x),I(\tau_*(x)+t,x,y)\biggl)=&(0,0)\\
		\di\lim_{\tau_*(x)\geq t\to+\infty}\biggl(T(\tau_*(x)-t,x),I(\tau_*(x)-t,x,y)\biggl)=&(0,0).
		\end{array}
		\end{equation}
	\end{enumerate}	
\end{Prop}
This proposition is proved in Section~\ref{sec:proofSIRT}. We note that, for the true $SIR$ model, one has a much stronger result, as one may relate the maximum of $I$ to the maximum of the derivative of the one-dimensional Fisher-KPP wave. While we do have the existence of travelling wave here (see \cite{BRR4}), we do not have precise convergence results for the solution of the Cauchy Problem. This will be done in a future study.

\section{Effects of the line on the total number of infected per location}\label{sec:effects}

We recall that the total number of infected people at a given location $(x,y)$,
for both models without and with the road, is given by
$$I_{tot}(x,y)=S_0\big(1-e^{-\beta \tilde v_\infty(x,y)}\big),$$ 
where $\tilde v_\infty$ is either $v_\infty$ or $v_\infty^r$,
that is, (the $v$ component of) the unique positive, stationary solution 
of the problem. We know that $v_\infty$ and $v_\infty^r$ have the same limit at infinity:
$0$ if $R_0\leq1$ and the positive zero $v^*$ of $f$ if $R_0>1$.
%
%
What differs is the rate of decay towards the limit state.
Indeed, Theorem~\ref{thm:further}$(ii)$-$(iii)$ below  imply that
the decay of $v_\infty$ is 
$$\sqrt{-f'(0)/d}\quad\text{if }R_0\leq1,\qquad
\sqrt{-f'(v_*)/d}\quad\text{if }R_0>1,$$
whereas the next result shows that the one of $(u_\infty^r,v_\infty^r)$ is strictly slower.
\begin{Th}\label{thm:decay}
	%
	The stationary solution to the problem with the road~\eqref{e2.4} satisfies
	$$u_\infty^r(x)=
	\begin{cases}
	0 & \hspace{-7pt}\text{if }R_0\leq1\\
	\frac\nu\mu\, v_* & \hspace{-7pt}\text{if }R_0>1
	\end{cases}
	\;+\,e^{-\kappa(x)|x|},\qquad
	v_\infty^r(x,y)=
	\begin{cases}
	0 & \text{if }R_0\leq1\\
	v_* & \text{if }R_0>1
	\end{cases}
	\;+\,e^{-\lambda(x,y)|x|},$$
	with 
	$$\lim_{|x|\to\infty}\kappa(x)=\lim_{|x|\to\infty}\lambda(x,y)=
	a_*\geq0,\qquad
	\text{locally uniformly in $y\geq0$}.
	$$
	Moreover, $a_*=0$ if $R_0=1$, whereas
	$$0<a_*<	
	\begin{cases}
	\sqrt{\frac{-f'(0)}d} & \text{if }R_0<1\\
	\sqrt{\frac{-f'(v_*)}d} & \text{if }R_0>1.
	\end{cases}
	$$
\end{Th}

This result has the important consequence of showing that $v_\infty^r>v_\infty$ 
close to the $x$-axis and far away from the origin.
Namely, {the presence of the road increases the number of infected people far from the epicentre
	of the epidemic}.

Observe that the decay of $v_\infty$ is the natural one provided by the linearised equation 
(outside $\supp I_0$) and it is obtained through rather standard arguments.
On the contrary, the decay of $v_\infty^r$ is slower than the natural one.
The proof of Theorem~\ref{thm:decay}, presented in Section~\ref{sec:proofSIRT},
relies on some ideas from~\cite{BRR2}, the keystone consisting in the 
construction of suitable sub and supersolutions.

Let us sketch the argument. 
To start with, we look for a stationary solution to~\eqref{e2.4} as a perturbation of the limit at infinity:
$$u(x)=\frac{\nu}{\mu}v_*+h \tilde u(x),\qquad v(x,y)=v_*+h\tilde v(x,y),$$
where, for notational simplicity, we have set 
$v_*:=0$ in the cases $R_0\leq1$.
Dropping the $o(h)$ terms, outside the supports of
$I_0$ and $T_0$ we get the linearised system
\Fi{linearised}
\left\{ \begin{array}{lll} -D \partial_{xx} \widetilde{u}=  \nu
	\widetilde{v}(x,0)-\mu \widetilde{u}(x)& (x\in \mathbb {R})\\ -d\varDelta
	\widetilde{v}=f'(v_*)\widetilde{v}& (x\in \mathbb {R},\ y>0)\\ -d\partial_y
	\widetilde{v}(x,0)=\mu \widetilde{u}(x)- \nu \widetilde{v}(x,0)& (x\in \mathbb {R}).
\end{array} \right.
\Ff
Observe that $f'(v_*)\leq0$. 
We seek for solutions in the form
\Fi{expsol}
\tilde u(x)=e^{-ax},\qquad
\tilde v(x,y)=\gamma e^{-ax-by},
\Ff
with $a,b,\gamma>0$. 
After some computation, the problem reduces to the algebraic system
\Fi{algebraic}
\begin{cases}
	\displaystyle Da^2=\frac{\mu db}{db+\nu}\\
	\displaystyle a^2+b^2=\frac{-f'(v_*)} d\\
	\displaystyle \gamma=\frac{\mu}{db+\nu}.
\end{cases}
\Ff
If $R_0\neq1$, i.e.~$f'(v_*)<0$,
this system admits a unique positive solution, that we denote $(a_*,b_*,\gamma_*)$, whose
first two components are represented in Figure~\ref{fig:a*}.
\begin{figure}[ht]
	\begin{center}
		\includegraphics[height=5.5cm]{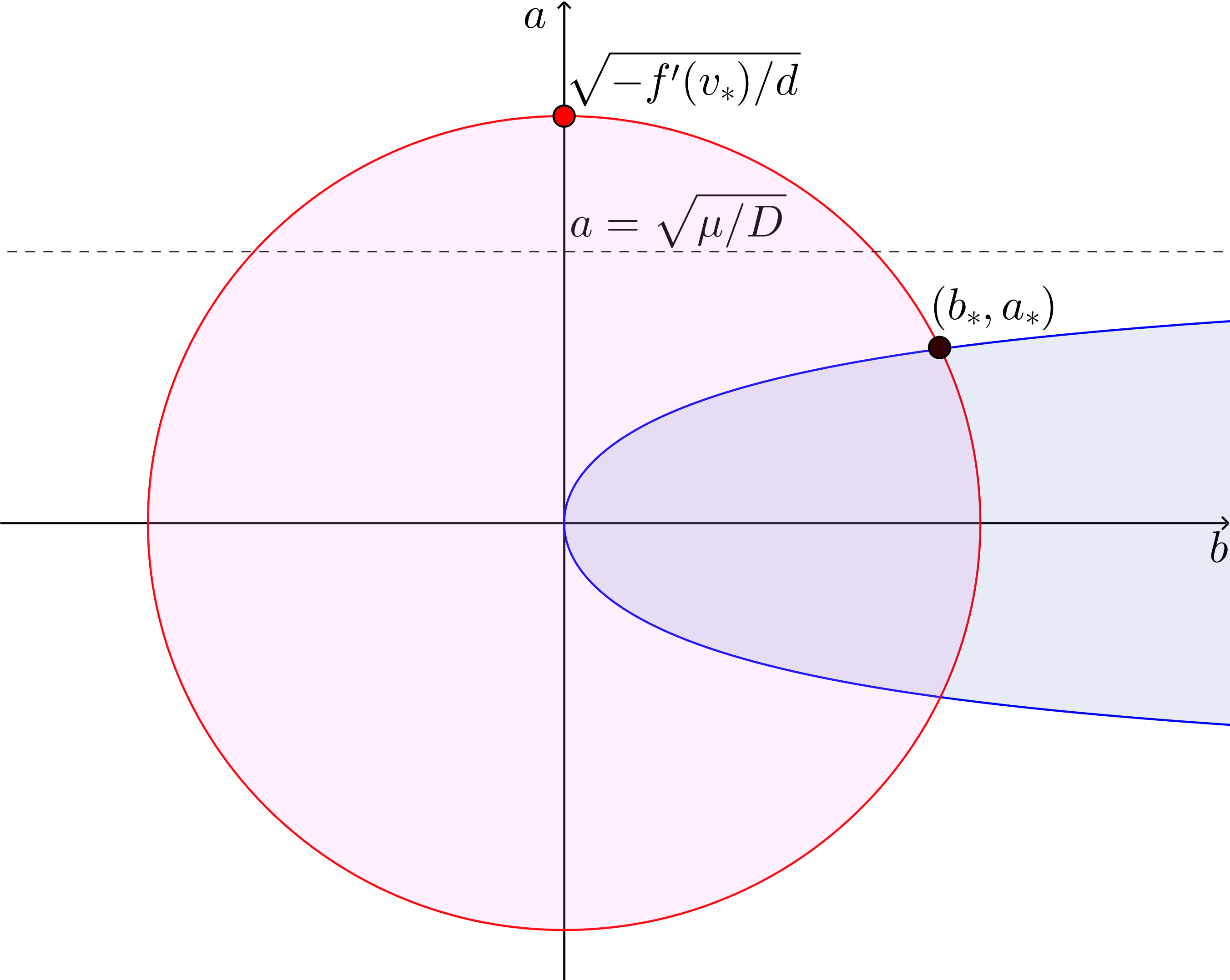}
		\caption{The solution $(b_*,a_*)$ to~\eqref{algebraic}. Shaded regions correspond to inequalities ``$\leq$''.}
		\label{fig:a*}
	\end{center}
\end{figure}
Observe that $a_*<\sqrt{-f'(v_*)/d}$, which is the radius of the disk.
If $R=0$, i.e.~$f'(v_*)=0$, then the unique solution of~\eqref{algebraic} is $(a_*,b_*,\gamma_*):=(0,0,\mu/\nu)$.
This will provide us with a supersolution to~\eqref{e2.4},
whence the exponential upper bound in Theorem~\ref{thm:decay}.

For the lower bound, we need to modify the construction in order to obtain 
a subsolution supported in a strip.
First of all, we penalise the reaction term in the field by considering a parameter
$\zeta>-f'(v_*)\geq0$. Next, we modify the above function $\tilde v$, considering now
the pair
\Fi{expsole}
\tilde u(x)=e^{-ax},\qquad
\tilde v(x,y)=\gamma e^{-ax}\big(e^{-by}-\e e^{by}\big),
\Ff
where $0<\e<1$. The linear system~\eqref{linearised} with $f'(v_*)$ replaced by $-\zeta$
reduces to the new algebraic system
\Fi{algebraice}
\begin{cases}
	\displaystyle Da^2=\frac{\mu db(1+\e)}{db(1+\e)+\nu(1-\e)}\\
	\displaystyle a^2+b^2=\frac\zeta d\\
	\displaystyle \gamma=\frac{\mu}{db(1+\e)+\nu(1-\e)}.
\end{cases}
\Ff
This system converges to~\eqref{algebraic} as $(\zeta,\e)\to(-f'(v_*),0)$, 
and thus its unique positive solution, denoted by
$(a_*^{\zeta,\e},b_*^{\zeta,\e},\gamma_*^{\zeta,\e})$, satisfies
\Fi{zetae}
\lim_{(\zeta,\e)\to(-f'(v_*),0)}(a_*^{\zeta,\e},b_*^{\zeta,\e},\gamma_*^{\zeta,\e})=(a_*,b_*,\gamma_*).
\Ff
This leads to a family of bounded subsolutions with a decay in the $x$-variable arbitrarily close to 
$a_*$.

The next result shows that there are regions where the road has the opposite effect.
\begin{Th}\label{thm:effect}
	Suppose that $R_0\neq1$ and that $I_0(x,y)$ is either symmetric with respect to~$y$, or satisfies 
	$\supp I_0\subset \RR\times[0,+\infty)$.
	Then
	there exist two subsets $\mc{E}^\pm$ of $\RR\times[0,+\infty)$
	such that
	$$v_\infty^r>v_\infty\quad\text{in }\mc{E}^+,\qquad
	v_\infty^r<v_\infty\quad\text{in }\mc{E}^-.$$
\end{Th}

This result shows that 
{the presence of the road decreases the number of infected people in some areas}.


%


\section{Proofs of the results}\label{sec:proofs}


\subsection{The $SIR$ model}\label{sec:proofSIR}

We give here the proofs of the results of Section~\ref{sec:SIR}, and we also derive some further qualitative
properties of the solution. We consider a more general framework:
we set problems~\eqref{eSIR},~\eqref{e2.2} in arbitrary space dimension $N\geq1$, i.e., $X\in\RR^N$,
and for initial data $v(0,X)$ not necessarily identically equal to zero, but rather
nonnegative and compactly supported.
This does not entail any additional difficulty.


\begin{proof}[Proof of Theorem~\ref{thm:Liouville}]
	On one hand, the function identically equal to $0$ is a subsolution to the stationary equation
	\Fi{KPPstationary}
	d\Delta v+f(v)+I_0(X)=0\quad(X\in\RR^N).
	\Ff
	On the other hand, the constant function $\bar v\equiv s$ is a supersolution of this equation provided
	$s>0$ is sufficiently large, because $I_0$ is bounded and $f(+\infty)=-\infty$.
	The existence of a solution $0\leq v_\infty \leq s$ then follows from the standard sub/supersolution method.
	Actually, as $0$ is not a solution, the elliptic strong maximum principle yields
	$v_\infty>0$ in $\RR^N$.
	
	Let us now derive the limit as $|X|\to\infty$.
	By elliptic estimates, any sequence of translations $v_\infty(X+X_n)$, with
	$(X_n)_{n\in\N}$ diverging, converges (up to subsequences) towards a nonnegative, bounded 
	solution $\tilde v$ to 
	$$d\Delta \tilde v+f(\tilde v)=0\quad(X\in\RR^N).$$ 
	If $R_0\leq1$, i.e.~$f'(0)\leq0$, then $f<0$ on $(0,+\infty)$, from which one readily
	derives $\tilde v\equiv0$
	(for instance by comparison with the ODE $\dot u=f(u)$) that necessarily. 
	This shows that $v_\infty(X)\to0$ as $|X|\to\infty$ when~$R_0\leq1$.  
	
	In the case $R_0>1$, we remark that $v_\infty$	is a supersolution to the classical Fisher-KPP equation
	\Fi{KPPstandard}
	\partial_t v=d\Delta v+f(v)\quad(t>0,\ X\in\RR^N),
	\Ff
	for which the ``hair-trigger'' effect holds, see~\cite{AW}: any solution with a positive, bounded initial datum
	converges as $t\to+\infty$, locally uniformly in space, to the positive zero~$v_*$ of $f$.
	We infer by comparison that $v_\infty\geq v_*$ in $\RR^N$.
	This shows in particular that $\tilde v$ is positive, and thus
	applying the ``hair-trigger'' effect to $\tilde v$ we derive $\tilde v\equiv v_*$.
	We have shown that $v_\infty(X)\to v_*$ as $|X|\to\infty$ when~$R_0>1$.

	To prove the uniqueness, we distinguish again the cases $R_0>1$ and $R_0\leq1$.
	
	\smallskip
	{\em Case $R_0>1$}.\\
	We need to show that, for any pair of positive, bounded solutions
	$v_1,v_2$ to~\eqref{KPPstationary}, there holds that $v_1\leq v_2$ in~$\RR^N$. 
	Assume by way of contradiction that 
	$$k:=\sup_{\RR^N}\frac{v_1}{v_2}>1.$$
	Since $v_1,v_2\to v_*$ at infinity, the 
	above supremum is a maximum, attained at some point~$\bar X$. Subtracting the equations 
	we get
	$$d\Delta (v_1-kv_2)+f(v_1)-kf(v_2)+I_0(X)(1-k)=0,$$ 
	which, evaluated at the point $\bar X$ (where $v_1=k v_2$ and
	$\Delta (v_1-kv_2)\leq0$) yields 
	$$k f(v_2(\bar x))\leq f(v_1(\bar X))=f(k v_2(\bar X)).$$
	But this is impossible because, being concave and vanishing at $0$, $f$ satisfies
	$$\forall s>0,\quad 
	\frac {f(ks)}{ks}<\frac{f(s)}s.$$
	\smallskip
	{\em Case $R_0\leq1$}.\\
	Now, any given positive, bounded solutions $v_1,v_2$ to~\eqref{KPPstationary}
	tend to $0$ at infinity. Take $\e>0$ and call $v_2^\e:=v_2+\e$.
	Because $f$ is decreasing in $\RR_+$, we see that
	$$d\Delta v_2^\e+f(v_2^\e)+I_0(X)<0\quad(X\in\RR^N).$$
	Assuming by contradiction that $v_1>v_2^\e$ somewhere and repeating the same arguments as before
	with $v_2$ replaced by $v_2^\e$, we end up with the inequality
	$$k f(v_2^\e(\bar X))<f(k v_2^\e(\bar X)),$$
	at some point $\bar X$ and for some $k>1$. As seen before, this is impossible.
	This means that $v_1\leq v_2^\e$ in $\RR^N$, which, by the arbitrariness
	of $\e$, entails the desired inequality $v_1\leq v_2$.	
\end{proof}

The following result contains some additional properties ov $v_\infty$.
\begin{Th}\label{thm:further}
	The stationary solution~$v_\infty(X)$ satisfies the following properties:
	\begin{enumerate}[$(i)$]

		%
		
		\item $v_\infty$ is radially decreasing outside the 
		support of $I_0$,
		i.e.,
		$$\forall e\in\mathbb{S}^1, \ \delta\leq r_1<r_2,\quad
		v_\infty(r_1 e)>v_\infty(r_2 e),$$
		where $\delta$ is such that $X\cdot e\leq\delta$ for all $X\in\supp I_0$;
		

		\item
		if $R_0>1$ then $v_\infty>v_*$ and moreover
		\Fi{decay>1}
		v_\infty(X)=v_* + e^{-\lambda(X)|X|},\quad
		\text{ with }\;\lim_{|X|\to\infty}\lambda(X)=\sqrt{\frac{-f'(v_*)}d};
		\Ff
		
		\item
		if $R_0\leq1$ then
		\Fi{decay<1}
		v_\infty(X)=e^{-\lambda(X)|X|},\quad
		\text{ with }\;\lim_{|X|\to\infty}\lambda(X)=\sqrt{\frac{-f'(0)}d}.
		\Ff	
		
	\end{enumerate}
\end{Th}

\begin{proof} \ $\ $\
	
	\smallskip
	{\em Statement $(i)$}.\\
	We make use of a reflection argument due to Jones~\cite{Jones}.
	Fix a direction $e\in\Sph$ and take $r>\delta$, where 
	$\delta$ is such that $X\cdot e\leq\delta$ for all $X\in\supp I_0$.
	Let $\mc{T}$ be the reflection with respect to the hyperplane $\{X\cdot e=r\}$, and define $\t v:=v_\infty\circ\mc{T}$. 
	In the half-space $\{X\cdot e<r\}$ ($\supset\supp I_0$) the function $\tilde v$ solves 
	$d\Delta\t v+f(\t v)=0$, hence it is a subsolution of the equation
	satisfied by $v_\infty$. In addition, $\t v\equiv v_\infty$ on~$\{X\cdot e=r\}$.
	Repeating the comparison argument in the proof of Theorem~\ref{thm:Liouville}
	with $v_1=\t v$ and $v_2=v_\infty$,
	but in $\{X\cdot e<r\}$, and observing that the point $\bar X$
	cannot belong to~the boundary $\{X\cdot e=r\}$, one infers that $\t v\leq v_\infty$ in $\{X\cdot e<r\}$.
	Actually, the strong maximum principle and Hopf's lemma imply that the inequality is strict and moreover
	$$\nabla \t v(re)\cdot e>\nabla v_\infty(re)\cdot e.$$
	Since $\nabla \t v(re)\cdot e=-\nabla v_\infty(re)\cdot e$, this
	gives the desired~monotonicity.
	
	\smallskip
	{\em Statements $(ii)$-$(iii)$}.\\
	We have shown in the proof of Theorem~\ref{thm:Liouville} that, in the case~$R_0>1$,
	$v_\infty\geq v_*$ in $\RR^N$. The strict inequality follows from the strong maximum principle.
	
	We simultaneously derive properties~\eqref{decay>1}-\eqref{decay<1}.
	To do so, we set $v_*=0$ in the case $R_0<1$.
	For given $e\in\Sph$, $k>0$ and $\lambda=\sqrt{-f'(v_*)/d}$,
	consider the function
	$w(X):=ke^{-\lambda X\cdot e}$.
	Using the concavity of $f$, we derive
	$$-d\Delta w=f'(v_*) w>f(v_*+w),$$
	that is, $v_*+w$ is a supersolution to~\eqref{KPPstationary} outside $\supp I_0$.
	Let $\rho>0$ be such that $\supp I_0\subset B_\rho$.
	We choose 
	$$k:=e^{\lambda\rho}\max v_\infty,$$
	so that $v_*+w\geq v_\infty$ in $B_\rho\supset\supp I_0$.
	As a consequence, the function $\min\big(v_\infty,v_*+w\big)$
	is a generalized supersolution to~\eqref{KPPstationary}.
	It then follows from the sub/supersolution method, that there exists a solution
	$0\leq v\leq \big(v_\infty,v_*+w\big)$, which is actually positive by the strong maximum principle.
	Then $v\equiv v_\infty$ thanks to the Liouville result of Theorem~\ref{thm:Liouville}.
	We have thereby shown that 
	$$\forall |X|\geq\rho,\quad
	v_\infty(X)\leq v_*+kz(X)=v_*+ke^{-\sqrt{-f'(v_*)/d}\, X\cdot e}.$$
	This being true for all $e\in\Sph$, we obtain the upper bound
	\Fi{decay<}
	\forall |X|\geq\rho,\quad
	v_\infty(X)\leq v_*+ke^{-\sqrt{-f'(v_*)/d}\, |X|}.
	\Ff
	
	Let us derive the lower bound. Take 
	$$\lambda:=\sqrt{\frac{-f'(v_*)+2\e}d},\quad\text{ with $\e>0$},$$
	and define $z:=e^{-\lambda |X|}$. This function satisfies,
	for $|X|\geq\frac{d\lambda(N-1)}{\e}$,
	$$-d\Delta z=\bigg(f'(v_*)-2\e+ d\lambda\frac{N-1}{|X|} \bigg)z
	\leq \big(f'(v_*)-\e\big)z.
	$$
	It follows that, for $h>0$ sufficiently small,
	the function $v_*+h z$ is a subsolution
	to~\eqref{KPPstationary} for~$|X|\geq\frac{d\lambda(N-1)}{\e}$.
	Up to decreasing $h$ if need be, we further have that $v_*+hz<v_\infty$ for
	$|X|\leq\frac{d\lambda(N-1)}{\e}$, hence 
	$\underline v:=\max \big(v_\infty,v_*+hz\big)$ is a generalized subsolution
	to~\eqref{KPPstationary}. By the sub/supersolution method, there exists
	a solution $\underline v\leq v\leq s$, where $s$ is such that $f(s)<\max I_0$.
	Theorem~\ref{thm:Liouville} eventually yields 
	$v\equiv v_\infty$. 
	This shows the lower bound 
	\Fi{decay>}
	\forall |X|\geq \frac{d\lambda(N-1)}{\e},\quad
	v_\infty(X)\geq v_*+h e^{-\sqrt{\frac{-f'(v_*)+2\e}d}|X|}.
	\Ff
	
	Call 
	$$\lambda(X):=\frac{-\log(v_\infty(X)-v_*)}{|X|}.$$
	The estimates~\eqref{decay<}-\eqref{decay>}
	yield, for $|X|$ sufficiently large,
	$$
	\sqrt{\frac{-f'(v_*)}d}-\frac{\log k}{|X|}\leq
	\lambda(X)
	\leq \sqrt{\frac{-f'(v_*)+2\e}d}-\frac{\log h}{|X|},$$
	from which the limits in~\eqref{decay>1}-\eqref{decay<1} 
	follow due to the arbitrariness of $\e>0$.
\end{proof}

We now turn to the results about the Cauchy problem~\eqref{e2.2}, with an initial datum 
$v(0,X)$ nonnegative and compactly supported.

\begin{proof}[Proof of Theorem~\ref{thm:ltb}]
	Let  $\underline{v}$, $\overline{v}$ be the solutions to the Cauchy problem
	emerging from the initial data identically equal to $0$ and $s$ respectively,
	with $s>0$ large enough so that $f(s)+\max I_0<0$.
	We further take $s>\max v(0,\cdot)$. 
	The parabolic comparison principle yields
	$$\forall t>0,\ X\in\RR^N,\quad
	\underline{v}(t,X)\leq v(t,X)\leq\overline{v}(t,X).$$
	Since the initial data of $\underline{v}$, $\overline{v}$ are respectively a sub and 
	a supersolution of the problem, the comparison principle
	implies that $\underline{v}$, $\overline{v}$ are respectively increasing and decreasing in $t$
	and then, by parabolic estimates,
	they converge to two stationary solutions $0<\underline{V}\leq\overline{V}<s$.
	The convergences occur locally uniformly in space and hold true for the time derivatives.
	Theorem~\ref{thm:Liouville} implies that $\underline{V}\equiv\overline{V}\equiv v_\infty$.
	The proof is complete.	
\end{proof}

\begin{proof}[Proof of Theorem~\ref{thm:speed}]
	Take $\e\in(0,\cSIR)$ and consider a sequence $(t_n)_{n\in\N}$ diverging to~$+\infty$
	and a sequence $(X_n)_{n\in\N}$ in $\RR^N$ such that $|X_n|\leq(\cSIR-\e)t_n$.
	If $(X_n)_{n\in\N}$ is bounded, we already know from Theorem~\ref{thm:ltb} that 
	$v(t_n,X_n)-v_\infty(X_n)\to0$ as $n\to\infty$.
	Suppose now that $(X_n)_{n\in\N}$ diverges (up to subsequences).  
	Recall that $v$
	is a supersolution to the standard Fisher-KPP equation~\eqref{KPPstandard}, for which 
	spreading occurs with the asymptotic speed $2\sqrt{d f'(0)}$, that is exactly $\cSIR$.
	Because $v(1,\cdot)>0$, we infer that
	$$
	\liminf_{n\to\infty}
	\big(v(t_n,X_n)-v_\infty(X_n)\big))\geq v_*-v_*=0.$$
	To derive the upper bound, we consider the same function 
	$w(X):=ke^{-\lambda X\cdot e}$ as in the proof
	of Theorem~\ref{thm:further}$(ii)$, 
	with $\lambda=\sqrt{-f'(v_*)/d}$ and $e\in\Sph$, $k>0$. 
	We have seen that $v_*+w$ is a supersolution to~\eqref{KPPstationary} outside $\supp I_0$.
	We then take $k$ large enough, independently of $e$, so that
	$v_*+w(X)>v(t,X)$ for all $t>0$ and $X\in\supp I_0\cup \supp v(0,\cdot)$.
	Hence, by comparison, $v(t,X)<v_*+w(X)$ for all $t>0$, $X\in\RR^N$. 
	This being true for any $e\in\Sph$, with $k$ independent of $e$, yields
	$v(t,X)\leq v_*+ke^{-\lambda |X|}$, for all $t>0$, $X\in\RR^N$.
	It follows that
	$$
	\limsup_{n\to\infty}
	v(t_n,X_n)\leq \limsup_{n\to\infty}\big(v_*+ke^{-\lambda |X_n|}\big)=v_*.$$
	The proof of the first limit stated in the theorem is achieved. 
	
	Let us deal with the second limit. For
	$c=\cSIR=2\sqrt{d f'(0)}$, $\lambda=\frac{\cSIR}{2d}$ and any given $e\in\Sph$ and $k>0$, the function
	$w(t,X):=ke^{-\lambda(X\cdot e-c t)}$ satisfies
	$$\partial_t w-d\Delta w-f'(0)w = \big(c\lambda-d\lambda^2 - f'(0)\big)w=0.$$
	Hence, by the concavity of $f$, it is a supersolution to~\eqref{e2.2} outside $\supp I_0$.
	We choose~$k$ large enough, independently of $e$, in such a way that
	$w(0,X)>v(t,X)$ for all $t>0$ and $X\in\supp I_0\cup \supp v(0,\cdot)$.
	Hence, by comparison, $v(t,X)<w(t,X)$ for all $t>0$, $X\in\RR^N$, and therefore, letting $e$ vary in
	$\Sph$, we get
	$$v(t,x)\leq  ke^{-\lambda(|X|-\cSIR t)},$$
	which gives the desired limit.	
\end{proof}

Using the same reflection argument as in the proof of Theorem~\ref{thm:further},
one can show that 
the solution $v(t,X)$ to~\eqref{e2.2} is 
radially decreasing with respect to $X$ outside the support of $I_0$.

%
%
%
%
%
%
%
%
%

We conclude this section with a monotonicity result with respect to the diffusion coefficient $d$.
We assume now for simplicity that the initial datum $v(0,\cdot)$ is identically equal to $0$.

\begin{Prop}\label{pro:d}
	Under the assumption that $I_0$ is radially decreasing, the values $v_\infty(0)$ and $v(t,0)$, 
	for any $t>0$, are decreasing with respect to the diffusion coefficient~$d$.
\end{Prop}

\begin{proof}
	Let $v^1,v^2$ be the solutions 
	to~\eqref{e2.2} associated with the coefficients $d=d^1$ and $d=d^2$ respectively, with $0<d_1<d_2$.
	The function $\t v^j(t,X):=v^j(t,\sqrt{d^j}X)$, for $j=1,2$, satisfies
	$$\t v^j_t-\Delta \t v^j= f(\t v^j)+I_0(\sqrt{d^j}X)\quad(t>0,\ X\in\RR^2).$$
	Thus, owing to the monotonicity of $I_0$, the comparison principle yields 
	$\t v^1(t,X)>\t v^2(t,X)$ for all $t>0$, $X\in\RR^N$, which, evaluated at $X=0$, gives
	$v^1(t,0)>v^2(t,0)$.
	
	Moreover, because $v^j$ converges towards the steady state $v_\infty^j$ as $t\to+\infty$,
	thanks to Theorem~\ref{thm:ltb}, one further derives 
	$$v_\infty^1(\sqrt{d^1}X)=\lim_{t\to+\infty}\t v^1(t,X)\geq
	\lim_{t\to+\infty}\t v^2(t,X)=v_\infty^2(\sqrt{d^2}X).$$
	This inequality is actually strict due to the strong maximum principle.
	We infer that $v_\infty^1(0)>v_\infty^2(0).$
\end{proof}


\subsection{The $SIRT$ model}\label{sec:proofSIRT}

We will make repeatedly use of the weak and strong comparison principles for the road-field system.
They are provided by~\cite[Proposition 3.2]{BRR2} in the case $I_0\equiv0$,
and one can check that the presence of the bounded source term $I_0$ does not affect their proofs. 
When applied to a stationary subsolution
$(\underline u,\underline v)$ and supersolution $(\overline u,\overline v)$ satisfying
$(\underline u,\underline v)\leq(\overline u,\overline v)$, the strong comparison principle
implies that the inequality is strict unless
$(\underline u,\underline v)\equiv(\overline u,\overline v)$.
Here and in the sequel, inequalities are understood component-wise.

\begin{proof}[Proof of Theorem~\ref{thm:Liouville-SIRT}]
	We start with constructing a stationary supersolution to~\eqref{e2.4}.
	Take $0<\sigma<\sqrt{\alpha/d}$ and 
	define the functions
	\Fi{uv-supersolution}
	\overline u:=K\frac{2\nu+d\sigma}\mu,\qquad
	\overline v:=K\big(1+e^{-\sigma y}\big),
	\Ff
	where $K$ is a positive constant that will be fixed later.
	The pair $(\overline u,\overline v)$ satisfies the last equation of~\eqref{e2.4}.
	Moreover, we compute	
	$$-D \partial_{xx} \overline u-\nu \overline v(x,0)+\mu \overline u=2dK\sigma,$$
	$$-d\Delta \overline v-f(\overline v)\geq -Kd\sigma^2-S_0+K\alpha,$$
	We then choose~$K$ sufficiently large so that the above right-hand sides are larger than
	$\max T_0$ and $\max I_0$ respectively. Then $(\overline u,\overline v)$
	is a supersolution to~\eqref{e2.4}.
	
	The existence of a positive, stationary solution follows from the sub/supersolution method,
	applied with $(0,0)$ as a subsolution and 
	$(\overline u,\overline v)$ as a supersolution.
	Owing to the strong maximum principle, this provides us with a solution 
	$$(0,0)<(u_\infty^r,v_\infty^r)<(\overline u,\overline v).$$
	
	We derive the uniqueness result, as well as the limit at infinity, by distinguishing the cases
	$R_0\leq1$ and $R_0>1$.
	
	\smallskip
	{\em Case $R_0>1$}.\\
	The positive steady state $(u_\infty^r,v_\infty^r)$ is a supersolution to the problem with $I_0, T_0\equiv0$,
	which reduces to the system studied in~\cite{BRR2}. For such system, 
	we know from \cite[Theorem~4.1]{BRR2} that positive solutions converge as $t\to+\infty$ to the steady state
	$(\nu/\mu,1)v_*$ ($v_*=1$ for the $f$ in~\cite{BRR2}), whence by comparison,
	\Fi{uv>}
	(u_\infty^r,v_\infty^r)\geq(\nu/\mu,1)v_*.
	\Ff
	
	Consider a sequence $((x_n,y_n))_{n\in\N}$ in $\RR\times\RR^+$, with
	$(y_n)_{n\in\N}$ diverging to $+\infty$.
	Then, by elliptic estimates, as $n\to\infty$, $v_\infty^r(x+x_n,y+y_n)$
	converges locally uniformly (up to subsequences) towards a 
	solution $\t v$ of the equation
	$-d\Delta \t v=f( \t v)$ in $\RR^2$.
	Moreover, $\t v\geq v_*$ due to~\eqref{uv>}.
	Then, as seen in the Proof of Theorem~\ref{thm:Liouville}, we necessarily have $\t v\equiv v_*$.
	This proves the second limit of the theorem in the case $R_0>1$.
	
	Take now a diverging sequence 
	$(x_n)_{n\in\N}$ in $\RR$, and let
	$(\t u(x),\t v(x,y))$
	be the limit of (a subsequence of) $(u_\infty^r(x+x_n),v_\infty^r(x+x_n,y))$,
	whose existence is guaranteed by elliptic estimates up to the boundary.
	Thus, $(\t u,\t v)$ is a stationary solution of~\eqref{e2.4}
	with~$I_0, T_0\equiv0$, which is positive due to~\eqref{uv>}.
	It follows from~\cite[Theorem~4.1]{BRR2} that $(\t u,\t v)\equiv(\nu/\mu,1)v_*$.
	This proves the first limit stated in the theorem (the uniformity in $y$ following from the first limit).
	
	It remains to prove the uniqueness. Let $(u_1,v_1)$ and $(u_2,v_2)$ be two pairs of
	positive, bounded, stationary solutions to~\eqref{e2.4}.
	Assume by way of contradiction that 
	$$k:=\max\bigg(\sup_{\RR}\frac{u_1}{u_2}\,,\,\sup_{\RR\times\RR^+}\frac{v_1}{v_2}\bigg)>1.$$
	Because of the limits we have just proved, one of the following situations necessarily occurs:
	$$\max_{\RR}\frac{u_1}{u_2}=k,\quad\text{or}\quad
	\max_{\RR\times\RR^+}\frac{v_1}{v_2}=k.$$
	Suppose we are in the latter case.
	Then, exactly as in the proof of ...,
	the concavity of $f$ prevents the maximum from being achieved in the interior of $\RR\times\RR^+$.
	Then, it is achieved at some point $(\bar x,0)$, and Hopf's lemma yields
	$$\partial_y(kv_2-v_1)(\bar x,0)>0.$$
	Using the third equation in~\eqref{e2.4}, together with $v_1(\bar x,0)=k v_2(\bar x,0)$, we find that
	$$ku_2(\bar x)=-\frac{kd}{\mu}\partial_y v_2(\bar x,0)+\frac{k\nu}{\mu}v_2(\bar x,0)
	<-\frac{d}{\mu}\partial_y v_1(\bar x,0)+\frac{\nu}{\mu}v_1(\bar x,0)=u_1(\bar x),$$
	which contradicts the definition of $k$.
	Consider the remaining case:
	$$\max_{\RR}\frac{u_1}{u_2}=k>\frac{v_1}{v_2}.$$
	Computing the difference of the equations satisfied by $ku_2$ and $u_1$ at a point 
	$\bar x$ where this maximum is achieved, we derive
	$$0\leq-D\partial_{xx} (ku_2-u_1)(\bar x)=
	\nu (kv_2-v_1)(\bar x,0)-\mu (ku_2-u_1)(\bar x)=\nu (kv_2-v_1)(\bar x,0).$$ 
	This is impossible because $v_1<kv_2$. 
	
	We have thereby shown that $k\leq1$, that is, $(u_1,v_1)\leq(u_2,v_2)$.
	Exchanging the roles of the solutions, yields the uniqueness result.
	
	\smallskip
	{\em Case $R_0\leq1$}.\\
	We start with the uniqueness result. We derive it in the more general framework
	of nonnegative solutions, with possibly $I_0\equiv0$.
	We need to show that, for any
	two pairs $(u_1,v_1)$, $(u_2,v_2)$ of nonnegative, bounded, stationary solutions to~\eqref{e2.4},
	there holds that $(u_1,v_1)\leq(u_2,v_2)$. Assume by contradiction that, on the contrary, 
	$$h:=\max\Big(\frac\mu\nu\sup_{\RR}(u_1-u_2)\,,\,\sup_{\RR\times\RR^+}(v_1-v_2)\Big)>0.$$
	
	Suppose first that $\sup_{\RR\times\RR^+}(v_1-v_2)=h$, and let $((x_n,y_n))_{n\in\N}$
	be a maximizing sequence. If $(y_n)_{n\in\N}$ is bounded from below away from $0$, then the functions 
	$v_j(x+x_n,y+y_n)$
	converge locally uniformly (up to subsequences) towards two 
	solutions $\t v_j$ of the equation
	$-d\Delta \t v_j=f( \t v_j)$ in a neighbourhood of the origin.
	Moreover, $(\t v_1-\t v_2)(0,0)=\max(\t v_1-\t v_2)=h$, and thus
	$$0\leq-d\Delta(\t v_1-\t v_2)(0,0)=f(\t v_1(0))-f(\t v_2(0))=f(\t v_2(0)+h)-f(\t v_2(0)).$$
	This is impossible, because $f'(0)=\alpha(R_0-1)\leq0$ and hence $f$ is decreasing on $\RR^+$.
	If, instead, $y_n\to0$ (up to subsequences), then the pairs
	$(u_j(x+x_n),v_j(x+x_n,y))$
	converge locally uniformly (up to subsequences) towards two 
	solutions $(\t u_j,\t v_j)$ of the same system, which is of the form~\eqref{e2.4} with $I_0$ either
	translated by some vector $(\xi,0)$, or replaced by $0$.
	Moreover, $(\t v_1-\t v_2)(0,0)=\max(\t v_1-\t v_2)=h$.
	If the maximum is also attained at some interior point, we get
	the same contradiction as before. Therefore, Hopf's lemma 
	yields
	$$0>d\partial_y(\t v_1-\t v_2)(0,0)=
	\nu(\t v_1-\t v_2)(0,0)-\mu(\t u_1-\t u_2)(0)=h\nu-\mu(\t u_1-\t u_2)(0).$$
	This contradicts the definition of $h$.
	
	Suppose now that
	$$h=\frac\mu\nu\sup_{\RR}(u_1-u_2)>\sup_{\RR\times\RR^+}(v_1-v_2).$$
	Considering now a maximizing sequence $(x_n)_{n\in\N}$ for $u_1-u_2$, then
	the limits (up to subsequences) $(\t u_j,\t v_j)$ of the translations 
	$(u_j(x+x_n),v_j(x+x_n,y))$, which, once again, satisfy a system analogous to~\eqref{e2.4}.
	The difference $\t u_1-\t u_2$ attains its maximum $\frac\nu\mu h$ at the origin, whence
	$$0\leq-D\partial_{xx} (\t u_1-u_2)(0)=
	\nu (\t v_2-\t v_1)(0,0)-\mu (\t u_2-\t u_1)(0)=\nu (\t v_2-\t v_1)(0,0)-\nu h.$$ 
	This contradicts $\sup_{\RR\times\RR^+}(v_1-v_2)<h$.
	
	The proof of the uniqueness is concluded. Let us pass to the limits at infinity.
	The limit $v_\infty^r(x,y)\to0$ as $y\to+\infty$, uniformly with respect to $x$,
	follows from the negativity of $f$, exactly as
	in the above proof of Theorem~\ref{thm:Liouville}.
	Consider now a diverging sequence $(x_n)_{n\in\N}$ in $\RR$.
	The sequence of translations 
	$(u_\infty^r(x+x_n),v_\infty^r(x+x_n,y))$ converges locally uniformly (up to subsequences)
	towards a bounded, stationary solution $(\t u,\t v)$ to~\eqref{e2.4} with $I_0\equiv0$.
	We have seen above that the Liouville-type result holds for nonnegative stationary solutions
	of such system, hence, necessarily, $(\t u,\t v)\equiv(0,0)$.
	This concludes the proof of the theorem.	
\end{proof}

We now turn to the set of results on the solution of the Cauchy problem~\eqref{e2.4}-\eqref{u0v0}.

\begin{proof}[Proof of Theorem~\ref{thm:ltb-SIRT}]
	Since the initial datum $(0,0)$ is a subsolution to~\eqref{e2.4} which is not a solution, 
	the comparison principle
	implies that the solution $(u,v)$ is strictly increasing in $t$.
	It further implies that $(u,v)$ is smaller than the supersolution constructed 
	in the proof of Theorem~\ref{thm:Liouville-SIRT}, defined by~\eqref{uv-supersolution}.	
	It follows that, as $t\to+\infty$, $(u,v)$ converges locally uniformly to a positive, bounded, 
	stationary solution, which necessarily coincides with $(u_\infty^r,v_\infty^r)$ due to 
	Theorem~\ref{thm:Liouville-SIRT}.
\end{proof}

\begin{proof}[Proof of Theorem~\ref{thm:speed-SIRT}]
	In the case $I_0,T_0\equiv0$, the result reduces to \cite[Theorem~1.1]{BRR2},
	with the only difference that the initial datum there must not be identically equal to $(0,0)$ (otherwise
	the solution remains $(0,0)$ for all times). 
	In that case, the limit state is simply $(u_\infty^r,v_\infty^r)\equiv(\nu/\mu,1)v_*$.
	We call $\cSIRT$ the speed provided by \cite[Theorem~1.1]{BRR2}. 
	
	Take $\e\in(0,\cSIRT)$ and consider a sequence $(t_n)_{n\in\N}$ diverging to~$+\infty$
	and a sequence $(x_n)_{n\in\N}$ in $\RR$ such that $|x_n|\leq(\cSIRT-\e)t_n$.
	If $(x_n)_{n\in\N}$ is bounded, then the convergence of 
	$(u(x_n),v(x_n,y))$ towards the steady state follows from Theorem~\ref{thm:ltb-SIRT}.
	
	Suppose that $(x_n)_{n\in\N}$ diverges (up to subsequences).  	
	By the strong maximum principle, the solution 
	$(u,v)$ is strictly larger than $(0,0)$ at, say, $t=1$.
	Fitting a compactly supported datum below it, and applying the spreading result
	from \cite[Theorem~1.1]{BRR2}
	to the solution of~\eqref{e2.4} with $I_0,T_0\equiv0$, emerging from such datum,
	we infer by comparison that 
	$$
	\liminf_{n\to\infty}
	\bigg(\big(u(t_n,x_n),v(t_n,x_n,y)\big)-\big(u_\infty^r(x_n),v_\infty^r(x_n,y)\big)\bigg)\geq 
	\Big(\frac\nu\mu,1\Big)v_*-\Big(\frac\nu\mu,1\Big)v_*=0,$$
	where we have also used the limit given by Theorem~\ref{thm:Liouville-SIRT}.
	On the other hand, by comparison, $(u,v)\leq(u_\infty^r,v_\infty^r)$, for all $t\geq0$.
	Thus, the first limit in Theorem~\ref{thm:speed-SIRT} is proved.
	
	Let us derive the second limit. We restrict to $x>0$, the case
	$x<0$ being obtained by a specular argument.
	We recall how the asymptotic speed --~named here $\cSIRT$~-- 
	is obtained in~\cite{BRR2}: it is
	the least $c$ so that the linearised form of~\eqref{e2.4} around $(0,0)$, 
	i.e.~when $f(v)$ is replaced by $f'(0)v=\alpha(R_0-1)v$
	and when $I_0$, $T_0$ are set to 0, admits plane wave solutions of the form
	$$
	\big(\phi(t,x),\psi(t,x,y)\big)=e^{-a(x-ct)}(1,\gamma e^{-b y}),\ \ (a,b,\gamma)>(0,0,0).
	$$
	Thus the triple $(c,a,b)$ solves the algebraic system (the coefficient $\gamma$ is easily computed from the exchange condition):
	\begin{equation*}
	\begin{cases}
	-Da^2+ca+\di\frac{db\mu}{\nu+db}=0\\
	-d(a^2+b^2)+ca=\alpha(R_0-1).
	\end{cases}
	\end{equation*}
	Hence, for $c=\cSIRT$, the pair $(\phi,\psi)$ above is a solution to the linearised system and therefore,
	by the concavity of $f$, it is a supersolution
	to the original system~\eqref{e2.4} outside the supports of $I_0$, $T_0$.
	The second limit in Theorem~\ref{thm:speed-SIRT} for $x>0$ then follows by comparison
	with $k(\phi,\psi)$, with $k$ sufficiently large so that $k(\phi,\psi)\geq 
	(u_\infty^r,v_\infty^r) (\geq (u,v))$ inside the supports of $I_0,T_0$.
\end{proof}

\begin{proof}[Proof of Proposition~\ref{p4.2}] 
	For $x\in\RR$, define $\tau_*(x)$ as the first time $t$ such that 
	\begin{equation*}
	\label{e6.1001}
	v(t,x,0)=\frac{v_*}2.
	\end{equation*}
	This is possible due to Theorem~\ref{thm:speed-SIRT} and~\eqref{uv>}.
	These also entail that the function~$\tau_*(x)$ is locally bounded and satisfies~\eqref{tau}.
	It is also clear that $\inf\tau>0$, because $v$ identically vanishes at $t=0$ and it
	is uniformly continuous by parabolic estimates.
	Assume \eqref{e4.100} to be false. Namely, there exists a sequence 
	$(x_n)_{n\in\N}$ in $\RR$ and a bounded sequence $(y_n)_{n\in\N}$ in $[0,+\infty)$
	such that one of the following situations occurs:
	\begin{equation}
	\label{e6.1000}
	\lim_{n\to\infty}I(\tau_*(x_n),x_n,y_n)=0,\quad\text{or}\quad
	\lim_{n\to\infty}T(\tau_*(x_n),x_n)=0.
	\end{equation}
	Because $\tau(x)$ is bounded from below away from zero and bounded from above,
	and $I$,~$T$ are positive for $t>0$,
	we necessarily have that $(x_n)_{n\in\N}$ diverges.
	Set 
	$$
	\begin{array}{rll}
	T_n(t,x):=&T(\tau_*(x_n)+t,x_n+x),\ \ I_n(t,x,y):=I(\tau_*(x_n)+t,x_n+x,y),\\ 
	S_n(t,x,y):=&S(\tau_*(x_n)+t,x_n+x,y);
	\end{array}
	$$
	from parabolic estimates, a subsequence of $(S_n,I_n,T_n)_{n\in\N}$ 
	(that we may assume without loss of generality to be the whole sequence) converges to an entire
	(i.e., for all $t\in\RR$) solution 
	$(S_\infty,I_\infty,T_\infty)$ of the $SIRT$ system~\eqref{e1.1}. 
	We may also assume $(y_n)_{n\in\N}$ to converge to 
	some $y_\infty\geq0$; by~\eqref{e6.1000} we have that either
	$I_\infty(0,0,y_\infty)=0$ or 
	$T_\infty(0,0)=0$. 
	In the latter case we deduce from the last equation in~\eqref{e1.1} that~$T_\infty\equiv0$
	and then~$I_\infty\equiv0$; in the former case the same conclusion follows from
	the strong maximum principle and the Hopf lemma, using the first and third equations in~\eqref{e1.1}.
	Set 
	$$
	u_n(t,x)=u(\tau_*(x_n)+t,x_n+x),\ \ v_n(t,x,y)=v(\tau_*(x_n)+t,x_n+x,y);
	$$
	as $u_n=\partial_tT_n$ and $I_n=\partial_tv_n$, a subsequence (that we may once again to be the whole sequence) converges to a $t$-independent pair 
	$(u_\infty(x),v_\infty(x,y))$, as their time derivatives converge to 0 as $n\to\infty$. So, $(u_\infty,v_\infty)$  is a stationary solution of~\eqref{e2.4} with $I_0,T_0\equiv0$. 
	But we know from~\cite[Theorem~4.1]{BRR2} that the Liouville-property
	holds for such system, that is, either $(u_\infty,v_\infty)\equiv(0,0)$ or
	$(u_\infty,v_\infty)\equiv(\frac\nu\mu ,1)v_*$.
	This is impossible because
	$$v_\infty(0,0)=\lim_{n\to\infty}v(\tau_*(x_n),x_n,0)=\frac{v_*}2,$$
	by the definition of $\tau_*(x)$. 
\end{proof}

We now prove the result about the rate of decay of the steady state.

\begin{proof}[Proof of Theorem~\ref{thm:decay}]
	We derive these limits in the case $x\to+\infty$, the limits at $-\infty$ being analogous.
	
	\smallskip
	{\em Upper bound.}\\
	%
	In order to simultaneously treat the cases $R_0\leq1$ and $R_0>1$,
	we set in the former $v_*:=0$. It follows that $f'(v_*)\leq0$ for all $R_0>0$. 
	Consider the pair $(\tilde u,\tilde v)$ given by~\eqref{expsol}, with 
	$(a,b,\gamma)=(a_*,b_*,\gamma_*)$ unique positive solution of~\eqref{algebraic} if $R_0\neq1$, 
	or $(a_*,b_*,\gamma_*)=(0,0,\mu/\nu)$ if $R_0=1$.
	Namely, $(\tilde u,\tilde v)$ satisfies the linearised problem~\eqref{linearised}. 
	By the concavity of $f$, we see that, for $h>0$,
	$$f(v_*+h\tilde v)\leq -\zeta h\tilde v,$$
	and therefore the pair $(\frac\nu\mu v_*+h\tilde u,v_*+h\tilde v)$ is a supersolution to~\eqref{e2.4}
	outside the supports of $T_0$,~$I_0$.
	We take $h$ large enough so that, inside these supports,
	such pair is larger than the supersolution
	$(\overline u,\overline v)$ used 
	in the proof of Theorem~\ref{thm:Liouville-SIRT}, defined by~\eqref{uv-supersolution}.
	As a consequence, the pair
	$$
	\Big(\,\min\Big(\frac\nu\mu v_* +h\tilde u\,,\,\overline u\Big)
	\,,\,\min\big(v_*+h\tilde v\,,\,\overline v\big)\Big)
	$$
	is a generalised supersolution to~\eqref{e2.4}.
	Therefore, we can find a stationary solution between $(0,0)$ and such supersolution, which
	necessarily coincides with $(u_\infty^r,v_\infty^r)$ thanks to the Liouville result
	of Theorem~\ref{thm:Liouville-SIRT}.
	We have thereby shown that 
	$$u_\infty^r\leq\frac\nu\mu v_* +h\tilde u,\qquad
	v_\infty^r\leq v_*+h\tilde v,$$
	and thus the desired upper bounds.
	
	\smallskip
	{\em Lower bound.}\\
	Fix $\zeta>-f'(v_*)\geq0$.
	We consider now the pair $(\tilde u,\tilde v)$ from~\eqref{expsole}, with 
	$0<\e<1$ and 
	$(a,b,\gamma)=(a_*^{\zeta,\e},b_*^{\zeta,\e},\gamma_*^{\zeta,\e})$ unique positive solution of~\eqref{algebraice}.
	The function $\tilde v(x,y)$ vanishes at $y=y^\e:=-\frac{\log\e}{2b_*^{\zeta,\e}}$, 
	it is bounded in the half-strip
	$$\mc{S}^\e:=\{x>0,\ 0<y<y^\e\},$$
	and, by construction, satisfies there $-d\Delta\tilde v=-\zeta \tilde v$.
	Hence, because $f(v^*)=0$ and $\zeta<-f'(v_*)$,
	the function $v^*+h \tilde v$ is a subsolution to the second equation of~\eqref{e2.4}
	in~$\mc{S}^\e$ provided $h>0$ is sufficiently small, depending on $\zeta$.  
	As a consequence, choosing $h=h^{\zeta}$ small, we have that the pair 
	$$(\underline u,\underline v):=\Big(\frac\nu\mu v_* +h\tilde u,v_*+h^{\zeta}\tilde v\Big)$$ is a 
	subsolution to~\eqref{e2.4} in~$\mc{S}^\e$.
	Up to replacing $h^{\zeta}$ with a smaller quantity $h^{\zeta,\e}$ if need be, 
	we can also require that 
	$$\underline u(0)\leq u_\infty^r(0),\qquad 
	\max_{y\in[0,y^\e]}\underline v(y)\leq \min_{y\in[0,y^\e]}v_\infty^r(y),$$
	and, in addition, that in the half-strip
	$\mc{S}^\e$, $(\underline u,\underline v)$ is smaller than
	the supersolution~$(\overline u,\overline v)$ defined by~\eqref{uv-supersolution}.
	Let us consider the solution $(u,v)$ of the evolution problem~\eqref{e2.4}
	having $(\overline u,\overline v)$ as initial datum.
	We know from the Liouville-type result 
	that $(u,v)\searrow (u_\infty^r,v_\infty^r)$ as $t\to+\infty$.
	It follows that, for all $t>0$,
	$$u(t,0)\geq u_\infty^r(0)\geq\underline u(0),\qquad
	\forall y\in[0,y^\e],\quad
	v(t,0,y)\geq v_\infty^r(0,y)\geq\underline v(0,y),$$
	and moreover $v(t,x,y^\e)>0=\underline v(x,y^\e)$ for $x\geq0$.
	We can therefore apply the comparison principle for the road-field system
	in the half-strip $\mc{S}^\e$ and deduce that, there, 
	$(u,v)$ remains larger than $(\underline u,\underline v)$
	for all $t>0$. Thus, we infer that
	$$\forall x>0,\quad
	u_\infty^r(x)\geq \frac\nu\mu v_*+h^{\zeta,\e} e^{-a_*^{\zeta,\e} x},$$	
	$$\forall x>0,\ y\in[0,y^\e],\quad
	v_\infty^r(x,y)\geq v_*+
	h^{\zeta,\e}\gamma_*^{\zeta,\e} e^{-a_*^{\zeta,\e} x}\big(e^{-b_*^{\zeta,\e} y}-\e e^{b_*^{\zeta,\e} y}\big).$$
	Owing to the arbitrariness of $\zeta>-f'(v_*)$ and $0<\e<1$, one then
	gets the desired lower bound using \eqref{zetae} and noticing that $y^\e\to+\infty$ as $\e\to0$.
	%
	%
\end{proof}

We conclude with the comparison between the steady state for the models without and with the
line.

\begin{proof}[Proof of Theorem~\ref{thm:effect}]
	The existence of the set $\mc{E}^+$, of the form $|x|\geq\rho$, $y<h$, directly follows from
	Theorems~\ref{thm:further} and~\ref{thm:decay}. Let us show the existence of $\mc{E}^-$.
	
	In the case where $I_0(x,y)$ is symmetric with respect to $y$, the stationary solution
	$v_\infty$ is symmetric with respect to
	$y$ too, by uniqueness.
	Instead, if $\supp I_0\subset \RR\times[0,+\infty)$ then we know that $v_\infty(x,y)$
	is decreasing with respect to $y$ in $[0,+\infty)$, for any $x\in\RR$. Then, in any case, $\partial_y v_\infty(x,0)\leq0$
	for all $x\in\RR$.
	Thus, integrating the equation $-d\Delta v_\infty=f(v_\infty)+I_0(x,y)$ on $\RR\times(0,+\infty)$ 
	yields
	$$\int_{\RR\times(0,+\infty)}\big(f(v_\infty)+I_0(x,y)\big)dx\,dy=
	d\int_{\RR}\partial_y v_\infty(x,0)dx\leq0.$$
	On the other hand, integrating the equations for $(u_\infty^r,v_\infty^r)$ 
	shows\footnote{\ The integrations are justified by Theorems~\ref{thm:further} and~\ref{thm:decay},
		together with elliptic estimates.}
	\[\begin{split}
	\int_{\RR\times(0,+\infty)}\big(f(v_\infty^r)+I_0(x,y)\big)dx\,dy &=
	d\int_{\RR}\partial_y v_\infty^r(x,0)dx\\
	&=\int_{\RR}\big(\nu v_\infty^r(x,0)-\mu u_\infty^r\big)dx=0.
	\end{split}\]
	Subtracting these two integrals we find that
	$$\int_{\RR\times(0,+\infty)}f(v_\infty)dx\,dy\leq
	\int_{\RR\times(0,+\infty)}f(v_\infty^r)dx\,dy.$$
	Recall that $v_\infty$ and $v_\infty^r$ are larger than $v_*$,
	and that $f$ is decreasing on $(v_*,+\infty)$. It follows that
	$$\int_{\mc{E}^+}f(v_\infty)dx\,dy>
	\int_{\mc{E}^+}f(v_\infty^r)dx\,dy,$$
	and therefore
	$$\int_{(\RR\times(0,+\infty))\setminus\mc{E}^+}f(v_\infty)dx\,dy<
	\int_{(\RR\times(0,+\infty))\setminus\mc{E}^+}f(v_\infty^r)dx\,dy.$$
	This in turn implies that there exists a set $\mc{E}^-\subset 
	(\RR\times(0,+\infty))\setminus\mc{E}^+$ where $v_\infty>v_\infty^r$.
	
\end{proof}


\section{The influence of $R_0$ and other parameters}\label{sec:parameters}
We wish to understand here how the different coefficients in our model will influence the speed $c_{SIR}^T$. For that we first write \eqref{e1.1} in non-dimensional form. We then write the algebraic system that leads to the 
(non-dimensional) velocity. Finally, we study how the reduced parameters will contribute to the enhancement of the $SIRT$ velocity. The space and time variables are non-dimensionalised as
\begin{equation*}
t=\frac\tau\alpha,\qquad (x,y)=\sqrt{\frac{d}\alpha}(\xi,\zeta),
\end{equation*}
while the unknowns $T(t,x)$ and $I(t,x,y)$ are expressed as
\begin{equation*}
T(t,x)=S_0\TT(\tau,\xi),\qquad I(t,x,y)=S_0\II(\tau,\xi,\zeta),
\qquad S(t,x,y)=S_0\SST(\tau,\xi,\zeta).
\end{equation*}
The parameters of importance are then found to be
\begin{equation*}
\DD=\frac{D}d,\qquad R_0=\frac{\beta S_0}\alpha,
\qquad \bar\nu=\frac\nu\alpha,\qquad \bar\mu=\frac\mu\alpha.
\end{equation*}
The speed $\cSIR$ is expressed as
\begin{equation*}
\cSIR=\sqrt{d\alpha}\wSIR,\qquad \wSIR=2\sqrt{R_0-1}.
\end{equation*}
System \eqref{e1.1} is then rewritten as
\begin{equation*}
\left\{ \begin{array}{lll} \partial_\tau \mathcal {I}-\varDelta \mathcal {I}+ \mathcal
	{I}=R_0 \mathcal {S}\mathcal {I}& (\tau>0,\ \xi \in \mathbb {R},\ \zeta>0)\\ \partial_\tau
	\mathcal {S}=-R_0 \mathcal {S}\mathcal {I}& (\tau>0,\ \xi \in \mathbb {R},\ \zeta>0)\\ -
	\partial_\zeta \mathcal {I}=\bar{\mu }\mathcal {T}-\bar{\nu }\mathcal {I}& (\tau>0,\ \xi \in
	\mathbb {R})\\ \partial_\tau \mathcal {T}-\mathcal {D}\partial_{\xi \xi }\mathcal
	{T}=\bar{\nu }\mathcal {I}(\tau ,\xi ,0)-\bar{\mu }\mathcal {T}& (\tau >0,\ \xi \in \mathbb {R}).
\end{array} \right.
\end{equation*}
The integrated quantities
$$
\UU(\tau,\xi)=\int_0^\tau\TT(\sigma,\xi))d\sigma,\qquad \VV(\tau,\xi,\zeta)=\int_0^\tau\II(t,x,y)ds.
$$
will then solve
\begin{equation}
\label{e.param.5}
\left\{ \begin{array}{ll} \partial_\tau \mathcal {U}-\mathcal {D}\partial_{\xi \xi }
	\mathcal {U}=  \bar{\nu }\mathcal {V}(\tau ,\xi ,0)-\bar{\mu }\mathcal {U}+\mathcal {T}_0(\xi )
	& (\tau>0,\ \xi \in \mathbb {R})\\ \partial_\tau \mathcal {V}-\varDelta \mathcal
	{V}=f(\mathcal {V})+\mathcal {I}_0(\xi ,\zeta )& (\tau>0,\ \xi \in \mathbb {R},\ \zeta>0)\\ -
	\partial_\zeta \mathcal {V}(\tau ,\xi ,0)=\bar{\mu }\mathcal {U}(\tau ,\xi )- \bar{\nu }\mathcal
	{V}(\tau ,\xi ,0)& (\tau >0,\ \xi \in \mathbb {R}). \end{array} \right.
\end{equation}
The function $f$ is given by 
$f(\VV)=R_0(1-e^{-\VV})-\VV$, so that, $f'(0)=R_0-1=\di\frac{\wSIR^2}4$. The initial quantities $\II_0$ and $\TT_0$ have obvious meanings. Now, recall that the  the minimal reduced speed for \eqref{e.param.5}, that we name $\wSIRT$, is shown to be the least $w$ so that the algebraic system in $a,b$
\begin{equation}
\label{e.param.7}
\left\{
\begin{array}{rll}
-\DD a^2+wa+\di\frac{\bar\mu b}{\bar\nu+b}=&0\\
-(a^2+b^2)+wa=\di\frac{\wSIR^2}4.
\end{array}
\right.
\end{equation}
has solutions. 
Let us discuss how the parameters $R_0$, $\bar\mu$, $\bar\nu$, $\DD$ interact so as to yield a large minimal speed $\wSIRT$, while $\wSIR$ is small. The latter condition just means that~$R_0$ is only slightly larger than 1. So, $\wSIR$ is now a small parameter. From the second equation of \eqref{e.param.7} we suspect that $a$, $b$ and $w$ will scale like $\wSIR$, while the first equation leads us to think that $a$   will additionally scale like $\DD^{-1/2}$, and that $w$ will scale like $\DD^{1/2}$. So, we perform a last round of scalings:
\begin{equation}
\label{e.param.10}
a=\frac{\wSIR}{\sqrt\DD}\bar a,\ \ b=\wSIR\bar b,\ \ w=\sqrt\DD\wSIR\bar w,
\end{equation}
and we introduce the new parameters
\begin{equation}
\label{e.param.9}
\lambda=\frac{\bar\mu}{\bar\nu\wSIR},\ \ \rho=\frac{\wSIR}{\bar\nu}.
\end{equation}
This leads to the final reduced system
\begin{equation}
\label{e.param.8}
\left\{
\begin{array}{rll}
-\bar a^2+\bar w\bar a+\di\frac{\lambda\bar b}{1+\rho\bar b}=&0\\
-(\di\frac{\bar a^2}{\DD}+\bar b^2)+\bar w\bar a=\di\frac14,
\end{array}
\right.
\end{equation}
Recall now that we are interested in a situation where $R_0$ is slightly above 1, that is, $\wSIR$ is small. From the scalings \eqref{e.param.10}, enhancement of propagation by the line should be achieved by a large reduced diffusion coefficient $\DD$. As the parameter $\rho$ is proportional to $\wSIR$, we anticipate that it will be small. Let $\WSIRT(\lambda)$ the minimal reduced speed in~\eqref{e.param.8}, we will look for it  in the limit $\DD\to+\infty$, $\rho\to0$ and $\wSIR\to0$.
This  amounts to estimating the minimal speed, still called $\WSIRT(\lambda)$ in the simplified system
\begin{equation}
\label{e.param.11}
\left\{
\begin{array}{rll}
-\bar a^2+\bar w\bar a+{\lambda\bar b}=&0\\
-\bar b^2+\bar w\bar a=\di\frac14.
\end{array}
\right.
\end{equation}
The first equation gives an inverted parabola $\Gamma_{1,\lambda,\bar w}$:
$$
\bar a=\frac{\bar w+\sqrt{\bar w^2+4\lambda\bar b}}2:=g(\bar w,\lambda,\bar b),
$$
starting from the point $(\bar b=0,\bar a=\bar w)$, while the second equation is the standard parabola $\Gamma_{2,\bar w}$
$$
\bar a=\frac1{\bar w}\biggl(\frac14+\bar b^2\biggl):=h(\bar w,\bar b).
$$
And so, we want to make $\Gamma_{1,\lambda,\bar w}$ and $\Gamma_{2,\bar w}$ intersect in the $(\bar b,\bar a)$ plane.
Given the behaviour of $g$ and $h$ for large $\bar b$, the other variables being fixed, we deduce that $\Gamma_{1,\lambda,\bar w}$ and $\Gamma_{2,\bar w}$ always intersect if $\bar w>\di\frac12$. This implies
$$
\WSIRT(\lambda)\leq\frac12.
$$
On the other hand, $\Gamma_{1,0,1/2}$ and $\Gamma_{2,1/2}$ intersect at their very start, that is $\bar b=0,\bar a=\di\frac14$, so that $\WSIRT(0)=\di\frac12$. Notice that one way to achieve  $\lambda=0$ is to have
$$
\wSIR\to0,\ \ \bar\nu\wSIR\to+\infty,\ \ \ \bar\mu=O(1),
$$
that is, a very high transmission from the domain to the line and a normal transmission from the line to the domain. More generally, one may also have
$$
\bar\mu\propto \wSIR^{\alpha},\ \bar\nu\propto\wSIR^{-1+\alpha},\ \ \ \ \hbox{for any $\alpha\in(0,1)$}.
$$
This corresponds to a small transmission from the road to  the field, yet a large transmission fromm the field to the road. This is sufficient to accelerate the epidemics.

Let us now increase $\lambda$, while keeping $\bar w\leq\di\frac12$. 
Since $g$ is strictly concave in $\bar b$, while $h$ is strictly convex in $\bar b$, the two graphs may have zero, one or two intersection points, the case of one intersection corresponding to the sought for reduced 
$SIRT$ velocity ${\WSIRT}(\lambda)$. Since 
$$
\frac{\partial}{\partial\lambda}\biggl(\frac{\partial g}{\partial{\bar b}}\biggl)=\frac{\bar w^2+4\lambda\bar b}{(\bar w^2+4\lambda\bar b)^{3/2}}>0,
$$
the function $\lambda\mapsto{\WSIRT}(\lambda)$ is strictly decreasing. One may also easily show that it tends to 0 as $\lambda\to\infty$. Summing up, we have proved the existence of a strictly decreasing function $\WSIRT$, with $\WSIRT(0)=\di\frac12$, tending to 0 at infinity, such that 
\begin{equation}
\label{e.param.12}
\lim_{_{\DD\to+\infty,\wSIR\to0}}\frac{\wSIRT}{\sqrt\DD\wSIR}=\WSIRT(\lambda),\ \ \ \lambda=\frac{\bar\mu}{\bar \nu \wSIR}.
\end{equation}
This means, in particular, that $\wSIRT$ can be quite large even if the reproduction number $R_0$ is close to 1. If such is the case, then $\wSIR$ is small. However, as we saw above, the speed $\wSIRT$ may be rendered large in several instances. 



\section{Discussion and conclusions}\label{sec:conclusions}

In this paper, we discuss the effects of the presence of a road on the spatial propagation of an epidemic within the context of a spatial $SIR$ model. The road has specific diffusion and infected can travel faster along it. To this end, we introduce a new model that we call a $SIRT$ model. In addition to the classical $S$, $I$ and $R$
compartments, it involves a compartment $T$ for travelling infected on this road. Here we only discuss the case of local Brownian diffusion and local interactions. 
In a forthcoming paper \cite{BRRnonloc}, we carry an analogous analysis for non-local interactions. 

By means of a classical transformation, this model can be reduced to a system involving a non-homogeneous Fisher-KPP type equation. 
The unknown functions for this system are
$$
u(t,x):=\int_0^t\TI(s,x)ds,\qquad v(t,x,y):=\int_0^tI(t,x,y)ds.
$$
This allows us to extend previous works and to derive some rather precise properties of this model. 
The main outcomes of our work are the following.
\begin{enumerate}
	\item We first show that the $SIRT$ system in the $(u,v)$ unknowns admits
	a unique positive (bounded) steady state, which describes the long-time behaviour of the solutions of the evolution system. When $R_0\leq 1$, this steady state 
	tends to 0 at infinity. We interpret this as saying that the epidemic does not propagate. On the contrary, when $R_0>1$, the steady state converges to some positive constant and the epidemic has propagated. 
	Thus, the position of~$R_0$ relative to 1 still governs the propagation or dying out of the epidemic.
	At this stage, the scenario is exactly the same as for the standard $SIR$ model.
	
	\item We compute an {\em asymptotic speed of propagation} for this model,
	that we call $\cSIRT$, and compare it 
	with $\cSIR$, the speed of propagation for the classical $SIR$ model with diffusion~\eqref{eSIR}.
	We show that $\cSIRT>\cSIR$ if  $D >2d$, where $D$ and $d$ are the diffusion 
	coefficients on the road and in the rest of the territory respectively.  In the case $D\leq 2d$, 
	then,  $\cSIRT=\cSIR$.
	Thus, the presence of a road with fast diffusion enhances the speed of propagation of the epidemic.
	
	\item We show that the $SIRT$ system is governed by four parameters: the basic reproduction number~$R_0$ (from which the classical $SIR$ speed of propagation is computed), the reduced transmission coefficients $\bar\mu$ and $\bar\nu$, and the ratio $\mathcal{D}= D/d$ between the diffusion on the road and the diffusion in the field. We find that, 
	even if $R_0$ is very close to $1$,
	the diffusion on the road may trigger a wave of contamination spreading at high speed. Even though the growth of the infection at each location may be slow, the propagation along the road may be fast. This may lead to situations where an epidemic may seem dormant and thus innocuous while it spreads seeds far away bringing about outbreaks and clusters apparently unconnected to the regions with a significant prevalence.  
	
	\item  We compare the total cumulative number of infected individuals per location, $I_{tot}(x)$, 
	in the cases with and without the road. We show that, compared with
	the standard $SIR$ model, the $I_{tot}(x)$ in the presence of the road 
	is larger in the range of $x$ large, that is, far from the epicentre 
	of the epidemic.
	This result is not intuitive a priori. Indeed, from the enhancement of the speed of spreading of the 
	epidemic wave by the road, it also follows that at any location the epidemic peak lasts less than without this enhancement. Therefore, one might have thought  that, moving faster, the total number of infected by the epidemic would go down. However, far from the epicentre, the contrary happens. 
	We also prove that while the total number of infected is higher far away, 
	there is also a region~$\mc{E}^-$, presumably close to the epicentre, where the total number of infected $I_{tot}(x)$ for the model with the road
	is smaller than the corresponding one for the standard $SIR$ model.
	It would be interesting to characterise such a set~$\mc{E}^-$ in some specific cases,
	establishing for instance whether it actually contains the epicentre of the epidemic.
	We leave this as an open question.
	
	The lower number of infected near the epicentre due to the presence of the road
	could be related to another phenomenon that we observe on the standard $SIR$ model. 
	In Proposition~\ref{pro:d} above we prove that
	the quantity $I_{tot}(x)$ evaluated at the epicentre of the epidemic $x=0$ is a decreasing function
	of the diffusion coefficient~$d$. This reflects the fact that a higher diffusion coefficient 
	``scatters'' more quickly the infected individuals far from the epicentre.
	The same mechanism could be at work for our model $SIRT$, where the road allows for more infected individuals to move away from the centre. 
	
	\item In a general manner, the $SIRT$-type of models we introduce here adds a compartment in the population dynamics and fills a gap at an intermediate scale. Indeed, for instance, most models presently used in monitoring the COVID-19 epidemic rely on the detailed analysis of two types of networks: the global network of air travel or the microscopic socio-economic or family and schools network. As our analysis shows, the intermediate network of roads and railways play an important role in the spatial spreading of epidemics at levels of a region or a country.
	
\end{enumerate}

\medskip
This model and its generalisations open the way to many open problems. For instance, 
a similar study would have to be carried out when there is diffusion not only of the infected but also of the susceptibles. This would lead to a system:
\begin{equation}
\label{ed2}
\left\{ \begin{array}{ll} \partial_tI-d\varDelta I+\alpha I=\beta SI& (t>0,\
	x\in \mathbb {R},\ y>0)\\ \partial_tS - d_S \varDelta S =-\beta SI& (t>0,\ x\in \mathbb {R},\
	y>0)\\ -d\partial_yI=\mu T-\nu I& (t>0,\ x\in \mathbb {R},\ y=0)\\ \partial_t T
	-D\partial_{xx}T=\nu I(t,x,0)-\mu T& (t>0, \ x\in \mathbb {R},\ y=0). \end{array} \right.
\end{equation}

\bigskip
The model we have presented and analysed in this paper sheds light on the effect of a road within an environment of slow diffusion for the spreading of epidemics. It allows us to explain some observations and to uncover various effects. The simplified structure allows us to carry a fairly complete mathematical analysis. 
Yet, this model could lend itself to more practical developments. 
Indeed, involving a newtork of roads should yield more precise results. To take into account roads in a more realistic fashion in discrete models is an important perspective.




\end{document}